\documentclass[english,letterpaper,11pt,reqno]{amsart}

\usepackage{appendix}
\usepackage{amsbsy}
\usepackage{amsfonts}
\usepackage{amsmath}
\usepackage{amssymb}
\usepackage{amsthm}
\usepackage{graphicx}
\usepackage{ifthen}
\usepackage{textcomp}
\usepackage{enumitem, color, amssymb}
\usepackage{mathrsfs}
\usepackage[bookmarksnumbered,colorlinks]{hyperref}
\hypersetup{colorlinks=true, linkcolor=blue, citecolor=blue}
\usepackage{placeins}

\newtheorem {lemma} {Lemma} [section]
\newtheorem{thm}{Theorem}
\newtheorem{prop}[lemma]{Proposition}
\newtheorem{defn}[lemma]{Definition}
\newtheorem{cor}[lemma]{Corollary}
\theoremstyle{remark}
\newtheorem{remark}[lemma]{Remark}

\newcommand{\beqa}{\begin{eqnarray}}
\newcommand{\beq}{\begin{equation}}
\newcommand{\eeqa}{\end{eqnarray}}
\newcommand{\eeq}{\end{equation}}

\newcommand{\be}{\begin{equation}}
\newcommand{\ee}{\end{equation}}
\newcommand{\lb}[1]{\label{#1}}

\renewcommand{\Ref}[1]{(\ref{#1})}

\newcommand\ip[2]{g({#1},{#2})}
\newcommand{\lra}{\longrightarrow}

\newcommand{\RR}{\mathbb{R}}
\newcommand{\SSt}{\mathbb{S}^3}
\newcommand{\vep}{\varepsilon}
\newcommand\kk{{\boldsymbol k}}
\newcommand\kt{\tilde{{\boldsymbol k}}}

\newcommand\vv[1]{{\boldsymbol {\it #1}}}
\newcommand\xx{\vv{x}}
\newcommand\yy{\vv{y}}

\newcommand\ZZ{\vv{z}}
\newcommand\cd[2]{\nabla_{\!#1}{#2}}

\newcommand\tg{\tilde{g}}
\newcommand\bg{\bar{g}}

\newcommand\n{\nabla}
\newcommand\gK{g_{\scriptscriptstyle K}}
\newcommand\nK{\n^{\scriptscriptstyle K}}
\newcommand\gS{g_{\scriptscriptstyle S\!K\!R}}
\newcommand\iiX{\iota^{\scriptscriptstyle X}}
\newcommand\iik[1]{\iota^{\scriptscriptstyle {#1}}}
\newcommand\comma{\hspace{.2in},\hspace{.2in}}

\newcommand{\HH}{\mathcal{H}}
\newcommand{\VV}{\mathcal{V}}
\newcommand{\ta}{\tau}
\newcommand\om{\omega}
\newcommand\tT{{\boldsymbol t}}
\newcommand{\wht}{\widehat}
\newcommand{\al}{\alpha}

\newcommand{\kr}{\bar{\kk}}
\newcommand{\we}{\wedge}
\newcommand{\bk}{\hat{\kk}}
\newcommand{\sig}{\sigma}
\newcommand{\del}{\delta}
\newcommand{\ep}{\varepsilon}
\newcommand{\fr}{\frac}
\newcommand{\tc}{\tau_c}
\newcommand{\ii}{\iota}
\newcommand{\J}{J_{g,\kk,\tT}}
\newcommand{\JJ}{J_{g,\kk_\pm}}

\begin{document}

\title[]{K\"ahler metrics via Lorentzian Geometry in dimension four}
\author[]{Amir Babak Aazami and Gideon Maschler}
%\address{Dept. of Mathematics, Clark University, Worcester, MA 01610}
%\email{Aaazami@clarku.edu\,,\,Gmaschler@clarku.edu}
\address{Department of Mathematics and Computer Science\\ Clark University\\ Worcester, MA }
\email{Aaazami@clarku.edu\,,\,Gmaschler@clarku.edu}
%\email{gmaschler@clarku.edu}
\begin{abstract}
Given a semi-Riemannian $4$-manifold $(M,g)$ with two distinguished vector fields
satisfying properties determined by their shear, twist and various Lie bracket
relations, a family of K\"ahler metrics $\gK$ is constructed, defined on an open
set in $M$, which coincides with $M$ in many typical examples. Under certain
conditions $g$ and $\gK$ share various properties, such as a Killing vector field
or a vector field with a geodesic flow. In some cases the K\"ahler metrics are complete.
The Ricci and scalar curvatures of $\gK$ are computed under certain assumptions in terms of data
associated to $g$. Many examples are described, including classical spacetimes in warped
products, for instance de Sitter spacetime, as well as gravitational plane waves, metrics of
Petrov type~$D$ such as Kerr and NUT metrics, and metrics for which $\gK$ is an SKR metric.
For the latter an inverse ansatz is described, constructing $g$ from the SKR metric.
\end{abstract}
\maketitle
\thispagestyle{empty}

%---------------------------------------------------------
\section{Introduction}
\label{sec:intro}

In one of his open problem collections, S. T. Yau concludes a problem with the question
\[
\begin{aligned}
\text{``}&\text{{\em Can one go from complete K\"ahler manifolds}}\\[-2pt]
  &\text{{\em back to physically interesting spacetimes?"}}
\end{aligned}
\]
(\cite[Problem 89]{yau}). This study makes a contribution mostly in the opposite direction,
by constructing K\"ahler metrics from Lorentzian $4$-manifolds equipped
with associated data.

\vspace{.08in}
To be sure, Yau's question is written in the context of the effort to make sense of the transformation between metrics known as Wick rotation. In contrast, the construction method
presented in this study, which applies in arbitrary signature, is different, and more involved,
than Wick rotation. But it is invariantly defined, and its two variants can be carried out
on a variety of classical spacetimes, such as de Sitter, Kerr and gravitational plane waves.
For one class of K\"ahler metrics, which includes the extremal metric conformal to the
Page metric, we give a kind of inverse construction: the K\"ahler metric is the input data for an ansatz producing Lorentzian $4$-manifolds, for which the construction method recovers the K\"ahler metric on an open dense set.

\vspace{.08in}
In some cases the construction yields complete K\"ahler metrics, or ones that extend to a larger compact manifold. There are also curvature-distinguished K\"ahler metrics, such as K\"ahler-Einstein ones, that arise from this construction. Some of those are described in the sequel~\cite{am2} to this work.

\vspace{.08in}
It is well-known that a Lorentzian metric is never compatible with any given almost
complex structure. Thus relating notions of Lorentzian and complex geometry is not a
straightforward process. Attempts to achieve this date back at least to the $1960$'s, and one of
its most well-known outcomes is the invention of twistor theory. Connections relating Lorentzian geometry specifically to K\"ahler geometry have also been made, some focusing on analogous structures in the two geometries, partly based on considerations from spin geometry \cite{gau,apo,trau1,nur-tra}. Our construction links the two geometries more directly, and is given in terms of standard differential geometric data, but is closely related to some of the papers just cited, and especially to \cite{nuro}. Flaherty's classical manuscript \cite{flaherty} also attempts such a direct link, but ultimately
follows a different path.
%, and so we devote a few paragraphs to historical comments in Section~\ref{Kerr}'s introduction,
%describing his approach and its relations to ours.

\vspace{.08in}
In more detail, given an oriented four-manifold with a semi-Riemannian metric $g$ and two
distinguished vector fields $\kk$, $\tT$ satisfying certain ``admissibility" conditions
(Definition~\ref{adms}), we construct an integrable almost complex structure and a family of exact
symplectic forms. Fixing one such form, it will be compatible with this complex structure and thus yield a K\"ahler metric $\gK$ on some open set (see Theorem~\ref{ad-gen}
and Proposition %\ref{prop:DgK} and
\ref{for-Kerr}). In favorable cases, which include
almost all of our examples, this open set coincides with the whole manifold.
Such a K\"ahler metric is not always complete, but we do give complete examples
(see subsection~\ref{complt}).
%Even in the two
%examples where this is not the case, both relating to the Kerr metric, the boundary of
%this open set is defined by a natural geometric condition.

\vspace{.08in}
In the case $g$ is Lorentzian and $\kk$ is null with a geodesic flow, such symplectic forms
were defined by the first author in \cite{AA}, which constituted the original motivation for
this study.

\vspace{.08in}
Nondegeneracy of the symplectic form, and integrability of the almost complex structure
are related to invariants associated with $\kk$, namely its twist and shear, respectively,
first introduced in General Relativity \cite{sac}. We employ variants of these notions
defined relative to a direct sum decomposition of the tangent bundle.

\vspace{.08in}
The shear-free condition for a single vector field and its relation to integrability of
almost complex structures was studied in \cite{calped, cal}. But our shear condition is more
general, as we show in an example (see Section~\ref{sec:shrfl}). In the shear-free case,
our K\"ahler metrics are ambihermitian and at times ambiK\"ahler, in the sense of~\cite{acg}.

\vspace{.08in}
We actually give two theorems related to integrability. A Newman-Penrose version of the conditions of the second (Theorem~\ref{thm:KerrNUT}) is given in Formula (VIII.5) of \cite{flaherty},
suitable for the special case where $\kk$, $\tT$ are part of a null tetrad (a special type of frame).
In some of our examples the frame is not of this type (see also
Remark~\ref{nonttrd}).

%As noted there, nowhere vanishing of the twist operator of $\kk$ played a
%main role in guaranteeing nondegeneracy of these symplectic forms
%The twist operator is a so-called optical invariant, first introduced in General Relativity
%\cite{sac}. Unlike standard versions, our notion of an optical invariant will be defined
%relative to the direct sum decomposition of the tangent bundle.
%%more general than
%(see subsection \ref{rel-sr-tw}).

\vspace{.08in}

\vspace{.08in}
%The relation between the curvatures of $g$ and $\gK$ is not straightforward, even in
%special cases where the two metrics are related by a biconformal change \cite{dan}
%composed on a ``metric rotation" obtained by changing the sign of $g$ on $\tT$.
% For instance, even with
%the simplifying assumptions that $\kk$ and $\tT$ are $g$-orthogonal and have opposite %$g$-magnitudes, the two metrics are related by a biconformal change composed on a ``metric %rotation", the latter obtained by changing the sign of the $g$-magnitude of $\tT$.
%Biconformal changes have been considered in studies related to harmonic morphisms. For example, %\cite{dan} includes formulas for the relation between the Levi-Civita connections of the two %metrics, though not general formulas relating their curvatures. The latter are, however, %cumbersome to work with, and we chose not to employ them in this work. Instead,
Ricci curvature computations of some of these K\"ahler metrics appear mostly in \cite{am2},
but also in Section \ref{Ric-scal} under certain assumptions.

%, most crucially that the manifold is contained in the
%total space of a holomorphic line bundle, we give here formulas for the Ricci and scalar curvatures of %$\gK$ in terms of data associated with $g$, $\kk$ and $\tT$ (Section \ref{Ric-scal}).
%Additionally, a result in \cite{SC} allows us to relate
%For the type of admissible manifold where our vector fields have geodesic flow, we also relate %
%completeness of $\gK$ with the completeness of the integral curves of $\kk$ and $\tT$,
%when these have a geodesic flow (subsection \ref{comp}). Finally, we describe a case
%where $\gK$ is itself admissible in the above sense, and the whole construction can be iterated %indefinitely (subsection \ref{repeat}).

%Conversely, results from \cite{SC} also imply that unboundedness of the magnitude of one of these
%vector fields forces the incompleteness of $\gK$.

\vspace{.08in}
Lorentzian metrics inducing K\"ahler metrics which are SKR are given in Section~\ref{skr}.
The latter were first introduced in \cite{dm1}, in the context of
the classification of conformally-Einstein K\"ahler metrics (Section \ref{skr}).
%Examples
%of SKR metrics on compact manifolds exist, and, as mentioned above, for those the ansatz
%produces an admissible Lorentzian metric on an appropriate open and dense set.

\vspace{.08in}
In Section~\ref{warp} we show that a large class of Lorentzian warped products are admissible.
The same is shown in Section~\ref{sec:planewave} for gravitational plane waves.

%give a general construction of admissible
%Lorentzian metrics, given as a warped product with a one-dimensional base, where the fiber
%is any $3$-manifold possessing a geodesic vector field with certain prescribed optical invariants.
%A number of examples of this construction are then given, one of which is de Sitter spacetime,
%and another derived from a  pp-wave metric in dimension four.
%We then describe an example which is not a warped product,
%where the admissible Lorentzian metric is a gravitational
%plane wave (Section \ref{sec:planewave}).

%\enlargethispage*{2000pt}

\vspace{.08in}
For metrics of Petrov type~$D$, we give in Section \ref{Kerr} three examples in which the
theory is implemented and produces K\"ahler metrics: the Kerr metric, a class of NUT metrics
and a metric conformal to the Kerr metric. The first two of these examples
require a different variant of the construction of an associated K\"ahler metric.
%The NUT
%metric we study is the only one which is both globally hyperbolic and has an
%associated K\"ahler metric $\gK$ defined on the whole manifold.
The complex structure for the NUT metric was first described in \cite{flaherty}.
Dixon~\cite{dix} has recently studied another K\"ahler metric Wick-rotated from the
Kerr metric on a domain in Kerr spacetime, and showed it is ambitoric (see also \cite{al-sc}).

%\vspace{.08in}
%Various mechanisms in which one varies an admissible Lorentzian metric
%and still produces the same K\"ahler metric $\gK$  are given in Section~\ref{invert}.
%This raises the issue of determining the set of all possible admissible metrics
%giving rise to a given $\gK$. We do not pursue this question here. What is clear from
%our viewpoint is that the possibility of expressing $\gK$ via $g$ is realized
%in a large variety of examples. It also provides an effective means of calculation
%of first and second order properties of $\gK$, as presented here and in \cite{am2}.

%\subsection*{Acknowledgements}
%We thank Akito Futaki for the suggestion to phrase our findings
%in the general semi-Riemannian context, and Paolo Piccione
%for a suggestion to consider the issue of completeness.
%We also acknowledge and thank Vestislav Apostolov for reading
%and commenting on an earlier version of this work, especially
%with regard to matters related to ambihermitian metrics.

%----------------------------------------
\section{Shear and twist}\lb{sec:shear0}

In this preliminary section we introduce variants of the notions
of \emph{shear} and \emph{twist}, the optical invariants that will have a significant
role in what follows. After describing their expressions
in appropriate frames, we compare our version to the more standard
one for null vector fields with a geodesic or pre-geodesic flow.
We then describe a few known applications valid especially for $3$-manifolds,
which will be needed in Section~\ref{warp}.
%\pagebreak

\subsection{Relative versions of shear and twist}\lb{rel-sr-tw}
Our notions of shear and twist will differ somewhat
from their common usage in the Physics literature, and also from
mathematical references such as \cite{calped}.
The need for these atypical definitions arises, in small part, from
their application to vector fields on semi-Riemannian manifolds which
may not be null, or even of constant length. But more importantly, the difference
is attributed to the fact that we consider \emph{two} (pointwise linearly independent)
distinguished vector fields, rather than just one. Thus our shear and twist
will be defined not with respect to the orthogonal complement of a single
vector field, but {\em relative} to a decomposition of the tangent bundle into
an orthogonal direct sum of distributions, one of which is spanned by these
two vector fields.
%In this section we will explicitly use
%the terms {\em relative shear} and {\em relative twist} for these notions,
%whereas the corresponding standard notions of shear and twist will be described in subsection~\ref{sec:shear2}. But throughout the rest of the file we will just
%be writing shear/twist, always meaning their relative version.
Specifically, for a semi-Riemannian manifold $M$,
let
\be\lb{split}
TM=\VV\oplus\HH
\end{equation}
be an orthogonal decomposition of the tangent bundle into two (necessarily
nondegenerate) distributions $\VV$, $\HH$.
Let $\pi\colon TM\to\HH$ be the projection relative to this decomposition, and let $\nabla$ denote the Levi-Civita connection of $M$.

\begin{defn}
\label{def:relative}
Let $X$ be a nowhere-vanishing vector field taking values in $\VV$.
The {\em relative shear operator} and {\em relative twist operator} of $X$ are
defined, respectively, as the $\HH\to\HH$ operators
given by
\[
\begin{aligned}
&\text{relative shear: $\n^o\!X:=$ trace-free symmetric part of $\pi\circ\n X\big|_\HH$,}\\
&\text{relative twist: $\n^s\!X:=$ skew-symmetric part of $\pi\circ\n X\big|_\HH$,}
\end{aligned}
\]
where $\n X$ refers to the linear operator $v\mapsto\n_vX$ on the tangent bundle.  If $\n^o\!X$ or $\n^s\!X$ vanishes, then $X$ is \emph{shear-free} or \emph{twist-free}, respectively.
\end{defn}
These relative optical invariants will be applied throughout most of the paper, and we will
often omit the term ``relative" when using them.

In all our applications the rank of $\HH$ will be two.
In this case the following related entity is real-valued,
and will play an important role.
\begin{defn}
The {\em (relative) twist function} of $X$ is
\[
|\ii|=|\iiX|:=2\sqrt{\det(\n^s\!X)}.
\]
If $\iiX$ is nowhere vanishing, then the flow of $X$ is called \emph{everywhere twisting}.
\end{defn}
The reason for the notation $|\ii|$ is that the twist function is the absolute
value of a function $\ii$ defined in the next subsection with respect to an
orthonormal frame of $\HH$.

In our most common application, the manifold will be Lorentzian of dimension four,
and admit an almost complex structure. Then $\VV$ will be the complex span of some vector field $X$.
%As mentioned above, to ensure that \Ref{split} holds, one requires that $\VV$ be non-degenerate %at each point.
In this Lorentzian setting, \emph{we will always choose $\VV$ to be timelike, so that $\HH = \VV^{\perp}$ will be spacelike}, i.e. $g\big|_\HH$ will be positive definite.

\subsection{Frame representation}\lb{sec:shear1}

%A trivial fact that will nonetheless be of great value in our applications is that, as we
%will see below, the projection $\pi$ does not appear explicitly in the matrix %representation of the (relative) shear and twist operators with respect to a local %orthonormal frame for the (rank two) distribution $\HH$.

%In fact,
In the setting of the previous subsection, assume $\HH$ has rank two, and let
$\xx$, $\yy$ be an ordered orthonormal frame for $\HH$. Then at each point, the matrix
of $\pi\circ\n X\big|_\HH$ (like that of $\n X\big|_\HH$) with respect to $\{\xx,\yy\}$ is
given by
$$
\left[\pi\circ\n X|_\HH\right]_{\xx,\yy}\ =\
 %   \left[
      \begin{bmatrix}%{cc}
        \ip{\cd{\xx}{X}}{\xx} & \ip{\cd{\yy}{X}}{\xx}\\
        \ip{\cd{\xx}{X}}{\yy} & \ip{\cd{\yy}{X}}{\yy}\\
      \end{bmatrix}\cdot
%    \right],
$$
Thus the shear operator of $X$ is
\be\lb{shr0}
[\n^o\!X]_{\xx,\yy}=
\begin{bmatrix}%{cc}
        - \sigma_1 & \sigma_2\\
        \sigma_2 & \sigma_1\\
      \end{bmatrix},
\end{equation}
where the entries are the \emph{shear coefficients}
\be
\begin{aligned}
\sigma_1\ &:=\ %\frac{1}{4}(\mathfrak{L}_{\kk}g)(\yy,\yy) - \frac{1}{4}(\mathfrak{L}_{\kk}g)(\xx,\xx)\ =\
 \frac{1}{2}\Big[\ip{\cd{\yy}{X}}{\yy} - \ip{\cd{\xx}{X}}{\xx}\Big] = \frac{1}{2}\Big[g([X,\xx],\xx)-g([X,\yy],\yy)\Big],\\
\sigma_2\ &:=\ %\frac{1}{4}(\mathfrak{L}_{\kk}g)(\xx,\yy)\ =\
 \frac{1}{2}\Big[\ip{\cd{\yy}{X}}{\xx} + \ip{\cd{\xx}{X}}{\yy}\Big] = -\frac{1}{2}\Big[g([X,\xx],\yy)+g([X,\yy],\xx)\Big]\cdot\label{eqn:shear2}
\end{aligned}
\end{equation}
While these coefficients are frame-dependent, $\sigma_1^2+\sigma_2^2 = -\det\n^o\!X$ is an
invariant quantity.
%Another way to give the shear coefficients is in terms of Lie bracket relations:
%%Let $d_X$ denote differentiation in the direction of $X$.
%%As $g(\n_X\xx,\xx)=d_X(g(\xx,\xx))/2=0$, and similarly for $\yy$ replacing $\xx$,
%%and $g(\n_X\xx,\yy)+g(\xx,\n_X\yy)=d_X(g(\xx,\yy))=0$, we can also write
%%the shear coefficients as
%\be
%\begin{aligned}
%2\sigma_1\ &=\ +g([X,\xx],\xx)-g([X,\yy],\yy),\\
%2\sigma_2\ &=\ -g([X,\xx],\yy)-g([X,\yy],\xx).\label{eqn:shear3}
%\end{aligned}

%\end{equation}

The twist operator of $X$ is given in this frame by
\[
[\n^s\!X]_{\xx,\yy}\ =
\begin{bmatrix}%{cc}
         0 & \iiX/2\\
        -\iiX/2 & 0\\
      \end{bmatrix},
\]
where
\be
\label{eqn:twist0}
\iiX=\ii:=\ip{\cd{\yy}{X}}{\xx} - \ip{\cd{\xx}{X}}{\yy}=\ip{X}{[\xx,\yy]}.
\end{equation}
As with the shear coefficients,  $(\iiX)^2 = 4\det\n^s\!X$ is invariant, namely it is the square of the twist function.
\begin{remark}\lb{iota}
Starting from Section~\ref{cx-str} and throughout the paper,
the distribution $\HH$ will always be oriented. We will use this
to {\em fix the sign of $\ii$} by the convention that $\ii$ is always computed as in
\Ref{eqn:twist0} with respect to an {\em oriented orthonormal frame}, oriented to
agree with the orientation of $\HH$.
%This convention will also
%be employed on the (non-admissible) $4$-manifolds of Section~\ref{Kerr}.
\end{remark}

%\begin{defn}
%The absolute value of $\iiX$ in \eqref{eqn:twist0} is the \emph{(relative) twist function:}
%\[
%|\iiX|:=2\sqrt{\det(\n^sX)}.
%\]
%If $\iiX$ is nowhere vanishing, then the flow of $X$ is \emph{everywhere twisting}.
%\end{defn}
Note that relations \Ref{eqn:shear2} have interesting consequences.
For example, if $X$ is a vertical field for a Riemannian submersion, while
$\HH$  is its horizontal distribution, then $X$ is shear-free, because
the bracket of a vertical vector field with a horizontal one is vertical.

%---------------------------------------------------------
\subsection{Classical shear and twist}
\label{sec:shear2}

We briefly compare here our versions of relative twist and relative shear with more standard notions
(see e.g., \cite{thurston} and \cite[Chapter 5]{o1995}). These standard notions will generally
not be used further, except for a version which applies to Riemannian $3$-manifolds, see below.

First, for later purposes we give some terminology related to the equation
\be\lb{pr-gd}
\text{$\cd{\kk}{\kk} = \al\kk$ for some smooth function $\al$ on $M$.}
\end{equation}
\begin{remark}
If equation \Ref{pr-gd} holds, $\kk$ will be called {\em pre-geodesic};
if $\al$ is not identically zero it will be called {\em strictly pre-geodesic}, and {\em geodesic} if $\n_\kk\kk=0$.
\end{remark}
If $\kk$ is a null pre-geodesic field on a Lorentzian 4-manifold, its covariant derivative induces
an operator $D\colon \kk^{\perp}/\kk \lra \kk^{\perp}/\kk$ on a quotient of its orthogonal
space. Then the (non-relative) shear and twist of $\kk$ are defined as the trace-free symmetric, respectively antisymmetric components of this operator.
For the relative case, $\HH$ in Definition \ref{def:relative} is the image of a chosen embedding $\kk^\perp/\kk\to\kk^\perp$, and the definition depends, of course,  on this choice.
%which yields an isometry $(\kk^\perp/\kk,\bar{g})\to(\HH,g\big|_\HH)$.
%In the non-relative null pre-geodesic case, the optical
%invariants are independent of the choice of $\HH$; whereas they do, in general,
%and vary with $\HH$ in Definition \ref{def:relative}.

%To repeat our statement in the beginning of the section, from now on we will always
%make use {\em only} of relative shear/twist, but nonetheless, to shorten our terminology,
%we will always refer to these as shear/twist.
%\end{remark}
%\begin{remark}
%Throughout the paper, if equation \Ref{pr-gd} holds, $\kk$ will be called {\em pre-geodesic};
%if $\al$ is not identically zero it will be called {\em strictly pre-geodesic}, and {\em %geodesic} if $\n_\kk\kk=0$.
%\end{remark}

%---------------------------------------------------------
%\subsection{Twist and shear on a Riemannian $3$-manifold}
%\label{sec:shear}

Let $(N,g)$ be a Riemannian $3$-manifold, $\kk$ a vector field on $N$ of unit length, so that $\kk \oplus \kk^{\perp} = TN$.
%We then have an orthogonal decomposition of $TN$ into the pointwise span of $\kk$ and
%its orthogonal complement $\kk^{\perp}$.
The relevant bundle map is now
%\beqa
%\label{eqn:D}
$D\colon \kk^{\perp} \lra \kk^{\perp}$, $D(v):=\cd{v}{\kk}$,
%\hspace{.2in},\hspace{.2in}v \in \kk^{\perp}\ \mapsto\ D(v):=\cd{v}{\kk},
%\eeqa
which is well defined as $\kk$ has constant length\footnote{By ``length"
of $\kk$ we will almost always mean $g(\kk,\kk)$.}.
%Geometrically, it may be seen as
%a second order approximation, via the ``screen" $\kk_p^{\perp}$ at a point $p$ along the flow of $\kk$.
%The names shear and twist represent the distorsion effects of these respective operators on
%a typical disk in the screen centered at $p$.
%At each $p \in M$, we may visualize $\kk_p^{\perp}$ as a ``screen" showing us the cross section %of the flow of $\kk$ through $p$, and $D_p$ as indicating how this flow evolves due to the %curvature.  The reason for the latter is because $D$ is a second order approximation of the flow of %$\kk$, as follows.  Letting $(x^i)$ be normal coordinates centered at a point $p \in M$, a first %order Taylor expansion of $\kk$ in these coordinates yields, at any other point $q = %\text{exp}_p(x^i(q)\partial_i|_p)$ in the domain of $(x^i)$,
%\begin{equation}
%\label{eqn:normal}
%\kk_q\ \approx\ \Big(\kk^i(p)+(\nabla_{x(q)^j\,\partial_j|_p}\kk)^i\Big)\,\partial_i\big|_q.
%\end{equation}
%In other words, to first order, the difference in the components of $\kk_p$ and a ``nearby" $\kk_q$ %is given by $D_p$ (in fact because $\kk$ has constant length, one can show that $D_p$ is completely %determined by its restriction to $\kk_p^{\perp}$).  Now consider a small disk $C_p$ of $g$-radius %$\vep$ on the screen $\kk_p^{\perp}$:
%$$
%C_p=\{v \in \kk_p^{\perp} : \ip{v}{v} \leq \vep\}.
%$$
In terms of  an orthonormal frame $\{\xx,\yy\}$ of $\kk^{\perp}$,
%be an orthonormal frame, the linear map $D_p$ sends $\{\xx_p,\yy_p\}$ to %$\{\cd{\xx_p}{\kk},\cd{\yy_p}{\kk}\}$, thereby deforming the disk $C_p$.  Because the %``infinitesimal generators" $\kk_p$ are themselves first order approximations of the flow of $\kk$, %\eqref{eqn:normal} allows us to interpret this deformation of $C_p$ via $D_p$ as arising from the %change in the flow of $\kk$ due to the curvature of $(M,g)$.  It is precisely in this sense that the matrix of the endomorphism $D_p$ is the same as \Ref{eqn:matrix} (for $X=\kk$), and the
the shear and twist of $\kk$ are still given by formulas \Ref{eqn:shear2} and \Ref{eqn:twist0}.

\subsection{Preliminary applications in dimensions $3$ and $4$}

For a Riemannian $3$-manifold $(N,g)$ with $\kk$, $\xx$, $\yy$ as in the previous section, Frobenius' theorem implies that $\iik{\kk}$, given as in \eqref{eqn:twist0}
\[
\iik{\kk}=\ip{\kk}{[\xx,\yy]},
\]
vanishes identically if and only if the orthogonal complement $\kk^{\perp}$ of $\kk$ is integrable.
We will be interested in the diametrically opposed situation where the frame independent
twist function $|\ii|$ is nowhere vanishing, so that  $\kk^{\perp}$ is nowhere integrable:
at any $p \in M$, there is \emph{no} embedded submanifold $S$ containing $p$ such that $T_qS = \kk_q^{\perp}$ for all $q \in S$. We record this analysis together with a related result, proven in \cite{hp13}.
\begin{lemma}%[\cite{hp13}]
\label{lemma:Ray}
If a unit length vector field $\kk$ on a Riemannian $3$-manifold is complete, has geodesic flow, and  \emph{$\text{Ric}(\kk,\kk) > 0$}, then $\kk^{\perp}$ is nowhere integrable. The latter occurs if and only if $\kk$ is everywhere twisting.
\end{lemma}
%The reason is because any unit length vector field $\kk$ with geodesic flow on a Riemannian 3-manifold satisfies the following two differential equations:
%\beqa
%\kk(\text{div}\,\kk) &=& \frac{\ii^2}{2} - \frac{(\text{div}\,\kk)^2}{2} - 2|\sigma|^2 - \text{Ric}(\kk,\kk)\label{eqn:Ray1}\\
%\kk(\ii^2) &=& -2(\text{div}\,\kk)\,\ii^2.\label{eqn:Ray2}
%\eeqa
%We emphasize that because of the presence of the twisting function $\ii^2$, \eqref{eqn:Ray1} is more general than the well known equation for the evolution of $\text{div}\,\kk$ in the case when $\kk^{\perp}$ is integrable (i.e., when $\ii^2 = 0$),
%$$
%\kk(\text{div}\,\kk)=-\sum_{i=1}^2 \lambda_i^2 - \text{Ric}(\kk,\kk),
%$$
%where $\lambda_1,\lambda_2$ are the eigenvalues of the endomorphism $D$ in \eqref{eqn:D}.  Indeed, \eqref{eqn:Ray1} and \eqref{eqn:Ray2} are unique to dimension 3, and work in tandem as follows.  Suppose that $\text{Ric}(\kk,\kk) > 0$, and that $\ii^2(p) = 0$ at some point $p \in M$.  Then \eqref{eqn:Ray2} dictates that $\ii^2$ must vanish along the entire integral curve of $\kk$ through $p$, in which case \eqref{eqn:Ray1} along this curve reduces to
%$$
%\kk(\text{div}\,\kk)\ \leq\ - \frac{(\text{div}\,\kk)^2}{2}\cdot
%$$
%If the flow of $\kk$ is complete, then a standard argument implies that $\text{div}\,\kk = 0$ along this integral curve, which is incompatible with the condition that $\text{Ric}(\kk,\kk) > 0$.

%Next, recall that the vanishing of the shear operator or the divergence of a vector field $\kk$ %are often described by the statement that $\kk$ is shear-free or divergence-free, respectively.
Next, we record the following well-known lemma, omitting its
standard proof.
\begin{lemma}
\label{lemma:KVF}
A unit length vector field on a Riemannian $3$-manifold is a Killing vector field if and only if it is geodesic, divergence-free, and shear-free.
\end{lemma}
The last two lemmas will only be applied in Section \ref{warp}.

%Returning to the setting of a null vector field $\kk$ with pre-geodesic flow on a Lorentzian %$4$-manifold $(M,g)$, analogous results hold.
%Integrability of $\kk^{\perp}$ is still equivalent to the vanishing of the twist function of %$\kk$,
%Indeed, letting $\{\kk,\xx,\yy\}$ denote, as before, a local frame of the orthogonal complement %$\kk^{\perp}$, with $\xx,\yy$ being two orthonormal vector fields orthogonal to $\kk$, it follows %by Frobenius' theorem that $\kk^{\perp}$ is integrable if and only if
%\[
%g(\kk,[\kk,\xx])=g(\kk,[\kk,\yy])=\underbrace{\,g(\kk,[\xx,\yy])\,}_{\ii}=0.
%\]
%However, as $\kk$ is null and pre-geodesic, $g(\kk,[\kk,\xx]) = g(\kk,[\kk,\yy]) = 0$,
%as can be seen by writing the Lie brackets using the (torsion-free) Levi-Civita connection,
%and applying its compatibility with the metric.
%Therefore,
%just as in the three-dimensional Riemannian setting,
Similarly, \emph{On a Lorentzian $4$-manifold the integrability of $\kk^{\perp}$, for a null pre-geodesic vector field $\kk$, is completely determined by the vanishing of the twist
function $|\ii| = |g(\kk,[\xx,\yy])|$} (see \cite[Ch. 5]{o1995}).  Moreover, a well-known
result holds in analogy with Lemma \ref{lemma:Ray}; see, e.g. \cite{AA} for a proof.
\begin{lemma}
\label{lemma:Ray2}
If a null vector field $\kk$ with geodesic flow on a Lorentzian $4$-manifold is complete and \emph{$\text{Ric}(\kk,\kk) > 0$}, then $\kk^{\perp}$ is nowhere integrable.
The latter occurs if and only if $\kk$ is everywhere twisting.
\end{lemma}
\section{Almost complex structures and integrability}\lb{cx-str}

In this section we introduce a key component of this work, namely
an almost complex structure attached to a semi-Riemannian $4$-manifold
equipped with certain data involving two vector fields. We investigate
its integrability and related properties.
%We also introduce three
%closely related notions of an admissible manifold. For all of them,
%the almost complex structure will be integrable. They differ from
%each other in the characteristics of one of the distinguished vector fields
%on the manifold.

\subsection{Admissible almost complex structures}\lb{defJ}
Let $(M,g)$ be an oriented semi-Riemannian $4$-manifold,
with two vector fields $\kk_+$, $\kk_-$.  Let $\VV:=\mathrm{span}(\kk_+,\kk_-)$
denote the distribution spanned at each point $p$ by $\kk_+|_p$, $\kk_-|_p$.
We assume that
%the $2\times 2$ matrix
\be\lb{nnsing}
\text{$\kk_+$, $\kk_-$ are everywhere linearly independent,}
\end{equation}
%with constant entries which are nonpositive on the diagonal.
%nonpositive diagonal entries.
and
\be\lb{space}
\text{$\HH:=\mathrm{span}(\kk_+,\kk_-)^\perp$ is spacelike.}
\end{equation}
\begin{remark}\lb{re-phrs}
Note in particular that $\kk_\pm$  then have no zeros.
Condition \Ref{space} means that  $g\big|_\HH$
is positive definite at each point, in particular it is pointwise nondegenerate,
which also implies that $g\big|_\VV$ is pointwise nondegenerate.
This last condition together with \Ref{nnsing} are equivalent to
\be\lb{nnsing1}
\text{$A:=\begin{bmatrix}g(\kk_+,\kk_+) & g(\kk_+,\kk_-)\\
                      g(\kk_-,\kk_+) & g(\kk_-,\kk_-)\end{bmatrix}$ is
                      everywhere nonsingular,}
\end{equation}
Thus \Ref{nnsing} and \Ref{space} are equivalent to \Ref{nnsing1} and \Ref{space}.
Finally, if $g$ is Lorentzian, \Ref{space} is equivalent to
\be\lb{time}
\text{$\VV=\mathrm{span}(\kk_+,\kk_-)$ is timelike, }
\end{equation}
i.e. $g\big|_\VV $ has Lorentzian signature at each point.
In that case %\Ref{nnsing} and
\Ref{time} implies
\be\lb{G}
G:=\det(A)<0.
\end{equation}
\end{remark}

We consider an almost complex structure
%\footnote{i.e, an endomorphism of the tangent bundle
%whose square is minus the identity}
$J=\JJ$ on $M$ defined as follows. First, we set $J\kk_+:=\kk_-$,
$J\kk_-:=-\kk_+$ and extend these relations linearly on $\VV$.
Second, we note that the restriction of $g$ to $\HH$ %:=\mathrm{span}(\kk_+,\kk_-)^\perp$
is positive definite, and $\HH$ inherits an orientation because $M$ is oriented and the ordered pair $\kk_+$, $\kk_-$ induces an orientation on $\VV$. We thus define $J\big|_\HH$ to be the unique endmorphism of $\HH$ whose square is minus the identity, which is additionally an isometry of $g\big|_\HH$
%almost complex
%structure making $g|_\HH$ hermitian
%\footnote{Recall here that a metric $g$ is hermitian if $g(Ja,Jb)=g(a,b)$ for any vector fields $a$, $b$.},
and respects the orientation on $\HH$.
%such that $J$ respects the orientation on $M$.
Finally, we define $J$ by extending $J\big|_\VV$, $J\big|_\HH$ linearly
on $TM=\VV\oplus\HH$.
\begin{defn}\lb{J-adm}
An almost complex structure $J=\JJ$ on an oriented $4$-manifold
is called {\em admissible} if it is constructed as above, using
a semi-Riemannian metric $g$, and two vector fields $\kk_\pm$ satisfying
\Ref{nnsing}, \Ref{space}. If the integrability relations \Ref{Nij}
below also hold, $J$ will be called an {\em admissible complex structure}.
\end{defn}

\subsection{Integrability}
Integrability of an almost complex structure $J$ implies the manifold admits complex coordinates.
It is defined by the vanishing of the Nijenhuis tensor
\[
N(a,b) = [Ja, Jb] - J[Ja, b] - J[a, Jb] - [a, b].
\]
We now give sufficient conditions for integrability of an admissible almost complex structure.
\begin{thm}\lb{integ} An admissible almost complex structure $J=J_{g,\kk_\pm}$
%defined as in subsection \ref{defJ}
%under assumptions \Ref{nnsing}, \Ref{space},
is integrable if the following three conditions
hold:
\be\lb{Nij}
\mathrm{i})\text{ $[\kk_\pm,\Gamma(\HH)]\subset\Gamma(\HH),\qquad$   \quad $\mathrm{ii})$\ $J\n^o\kk_+=\n^o\!J\kk_+$ on $\HH$.}
\end{equation}
\end{thm}
%Note that expressions such as $g([\kk_\pm,\cdot],\kk_\pm)$ are linear over $\HH$.
The notation of i) means that the Lie bracket operation with $\kk_+$ or $\kk_-$ sends any vector field in $\HH$ to another such vector field. The shear notation in ii) is as in Definition~\ref{def:relative}.

Note that condition ii) is invariant under multiplying both vector fields $\kk_\pm$ by a
common factor, while conditions i) are not. This has the consequence that the conclusion of the
theorem still holds if conditions i) are replaced by the requirement that they hold instead
with $\kk_\pm$ replaced by some common nowhere vanishing multiple of themselves.
%Section \ref{sec:shear0}.
%\begin{remark}
%Note that if $J$ instead takes $\kk_+$ to a linear combination of $\kk_+$ and $\kk_-$
%(with coefficients which are {\em constant}, at least in directions
%tangent to $\HH$), then conditions i) imply the same conditions on the pair $\kk_+$, $J\kk_+$,
%so that integrability of this $J$ will still follow from the theorem.
%\end{remark}
\begin{proof}
We examine the Nijenhuis tensor of $J$
for a frame $\{\kk_\pm$, $\xx_\pm\}$, where
$\{\xx_+,\xx_-\!\!=\!\!J\xx_+\}$ is an oriented orthonormal frame for $g\big|_\HH$.
Clearly $N$ vanishes on any pair $a$, $Ja$. The relation
$N(a,b)=JN(a,Jb)$ along with the antisymmetry of $N(a,b)$
imply that it is enough to check the vanishing for the pair
$\kk_+$, $\xx_+$.

We thus analyze
\be\lb{Nijen}
\begin{aligned}
N(\kk_+,\xx_+)& = [J\kk_+, J\xx_+] - J[J\kk_+, \xx_+] - J[\kk_+, J\xx_+] - [\kk_+, \xx_+]\\
            & = [J\kk_+,   \xx_-] - J[J\kk_+,  \xx_+] - J[\kk_+,  \xx_-] - [\kk_+, \xx_+].
\end{aligned}
\end{equation}
%Because of conditions i) of \Ref{Nij}, the inner product of each of the four Lie brackets in \Ref{Nijen}
%with $\kk_+$ or $\kk_-$ is zero. In fact split such a bracket as a linear combination of
%(some of) the Lie brackets $[\kk_\pm,x_\pm]$. Take each of the latter bracket terms and expand it as a linear combination
%of the frame vector fields. Then take the inner product of each side of this equation with
%$\kk_+$ or $\kk_-$. Invoking \Ref{nnsing} for the resulting linear system shows that
%the coefficients of $\kk_\pm$ in such an expansion must vanish. Thus all the bracket
%terms in \Ref{Nijen} are in fact sections of $\HH$. As $\HH$ is $J$-invariant,
Conditions \Ref{Nij}i) along with the $J$-invariance of $\HH$ imply
via \Ref{Nijen} that $N(\kk_+,\xx_+)$ is a section of $\HH$.

Next, taking the inner product of the right-hand side of \Ref{Nijen} with $\xx_+$,
%while replacing each of its terms according to $[a,b]=\n_ab-\n_ba$ and
while employing the fact that $g\big|_\HH$ is hermitian, we arrive at the following expression:
\begin{equation*}
\begin{aligned}
g(N(\kk_+,\xx_+),\xx_+)
&=g([J\kk_+,\xx_-],\xx_+)+g([J\kk_+,\xx_+],\xx_-)\\
&+g([\kk_+,\xx_-],\xx_-)-g([\kk_+,\xx_+],\xx_+)
\end{aligned}
\end{equation*}
%\begin{equation*}
%\begin{aligned}
%g(N(\kk_+,\xx_+),\xx_+)
%&=g(\n_{J\kk_+}\xx_-,\xx_+)-g(\n_{\xx_-}J\kk_+,\xx_+)\\
%&+g(\n_{J\kk_+}\xx_+,\xx_-)-g(\n_{\xx_+}J\kk_+,\xx_-)\\
%&+g(\n_{\kk_+}\xx_-,\xx_-)-g(\n_{\xx_-}\kk_+,\xx_-)\\
%&-g(\n_{\kk_+}\xx_+,\xx_+)+g(\n_{\xx_+}\kk_+,\xx_+).\\
%\end{aligned}
%\end{equation*}
%The fifth and seventh term vanish while the first and third cancel.
Referring now to the shear coefficient expressions \Ref{eqn:shear2},
the above yields the following two equalities, the second obtained
in analogy with the first:
\[
\begin{aligned}
g(N(\kk_+,\xx_+),\xx_+)&= -2\sigma_1^{\kk_+}-2\sigma_2^{J\kk_+},\\
g(N(\kk_+,\xx_+),\xx_-)&= +2\sigma_2^{\kk_+}-2\sigma_1^{J\kk_+},\\
\end{aligned}
\]
where the shear-related notations are as in Section \ref{sec:shear1}.
Since the action of the shear matrix \Ref{shr0}
on each of the standard basis vectors in $\mathbb{R}^2$ yields
$(-\sig_1, \sig_2)$ and $(\sig_2, \sig_1)$, respectively,
the last two equations yield the invariant formula
\be\lb{Nkplus}
\ii_{\kk_+}N=2(\n^o\kk_+-\n^oJ\kk_+\circ J)\ \mathrm{on}\ \HH.
\end{equation}
%where the shear operator is considered
In fact, they yield equality of both sides on $\xx_+$, and we obtain it on $\xx_-$ because of both
$N(\kk_+,\xx_-)=-JN(\kk_+,\xx_+)$ and the fact that $J\big|_\HH$ anticommutes with any trace-free
symmetric operator $P$ acting on $\HH$. This last fact holds since $J$ makes $g\big|_\HH$ hermitian, so that the adjoint
of $PJ$ is $-JP$, while the trace-free condition implies $g(PJ\xx_+,\xx_-)=g(P\xx_-,\xx_-)=-g(P\xx_+,\xx_+)=g(PJ\xx_-,\xx_+)$,
so that $PJ$ is also self-adjoint.

By applying the last mentioned fact to \Ref{Nkplus}, the theorem follows.
\end{proof}
A generalization of this theorem will appear
in subsection \ref{slf-adj}.
\begin{remark}\lb{doubleJ}
It is easily seen that if $J$ satisfies the conditions of this theorem,
the almost complex structure defined just as $J$,
but with respect to the opposite orientation, will also be integrable
if $\kk_\pm$ are shear-free.
%This follows since condition i)
%of Theorem~\ref{integ} is independent of orientation, so still holds,
%while condition ii) follows because the new almost complex structure
%still sends $\kk_+$ to $\kk_-$ and the shears of these vector fields
%are independent of orientation (though not their representation in
%appropriate ordered frames). Thus in the equation $J\n^o\kk_+=\n^o\kk_-$,
%applied to vectors in $\HH$, only the left hand side acquires a minus sign
%when switching the almost complex structure, hence both sides remain zero.
\end{remark}

Note that \Ref{nnsing1} implies that \Ref{Nij}i)
is equivalent to the four conditions
\be\lb{alter}
\text{$g([\kk_\pm,\cdot],\kk_\pm)=0$ and $g([\kk_\pm,\cdot],\kk_\mp)=0$ on $\HH$.}
\end{equation}
The converse of Theorem~\ref{integ} does not hold in general. However, we have
\begin{prop}\lb{int-nec}
If three of the conditions \Ref{alter} hold, and $N=0$, then the fourth, along with ii) of \Ref{Nij}, also hold.
\end{prop}
\begin{proof}
Assume $N=0$ and, for example, $g([\kk_+,\cdot], \kk_+)=g([\kk_-,\cdot], \kk_-)=g([\kk_+,\cdot], \kk_-)=0$ on $\HH$. Let
\[
A=[g(\kk_\pm,\kk_\pm)]:=\begin{bmatrix} p & r\\ r & q\end{bmatrix},
\]
and express the following vector fields in our standard frame:\\
$[\kk_-,\xx_-]=a\kk_++b\kk_-+\cdots$,
$[\kk_-,\xx_+]=c\kk_++d\kk_-+\cdots$, for coefficients $a$, $b$, $c$, $d$.
The coefficients $a$, $b$ may be obtained by applying the inverse of $A$ to
the vector $(g([\kk_-,\xx_-],\kk_+),g([\kk_-,\xx_-],\kk_-))=(g([\kk_-,\xx_-],\kk_+),0)$,
while for obtaining $c$, $d$ one applies the same matrix to
$(g([\kk_-,\xx_+],\kk_+),g([\kk_-,\xx_+],\kk_-))=(g([\kk_-,\xx_+],\kk_+),0)$.
On the other hand similar coefficients for $[\kk_+,\xx_\pm]$ all vanish by our
assumptions. Substituting the above expressions for the Lie bracket terms in $N$,
we see that
\[
N(\kk_+,\xx_+)=(a+d)\kk_++(b-c)\kk_-+\text{terms in $\HH$}.
\]
As $N=0$, these coefficients of $\kk_\pm$ vanish, and together with the above method of obtaining
$a$, $b$, $c$, $d$ this gives the two equations $qB=rC$, $-rB=qC$,
where $B=g([\kk_-,\xx_-],\kk_+)$, $C=g([\kk_-,\xx_+],\kk_+)$. These equations in turn imply
$qr(B^2+C^2)=0$, which easily leads to $B=C=0$, since $q$ and $r$ cannot both vanish.
We thus see that $g([\kk_-,\cdot],\kk_+)$ vanishes on $\HH$.
Then (ii) of \Ref{Nij} follows as in Theorem \ref{integ}.
\end{proof}

\subsection{Geometric conditions implying \Ref{alter}}\lb{geo-cond}
We now consider the question of whether there
are geometric circumstances in which any one of
the four conditions \Ref{alter} holds automatically.
All of them, of course, must be satisfied simultaneously
for condition \Ref{Nij}i) to hold. The observations
we note here, for which we omit the straightforward proofs,
will serve to verify integrability
of admissible almost complex structures appearing
in our examples in Sections~\ref{warp}-\ref{sec:shrfl}.

\begin{remark}\lb{cond-i}
For $\kk=\kk_\pm$, relation $g([\kk, \cdot],\kk)=0$ holds on $\HH$ in two cases:
\be\lb{double}
\begin{aligned}
\text{a)  }&\text{$\kk$ is a pre-geodesic vector field of constant length, or}\\
\text{b) }&\text{$\kk$ is Killing.}
\end{aligned}
\end{equation}
%The first of these follows as
%$g([\kk,\xx],\kk)=g(\n_{\kk} \xx-\n_{\xx}\kk,\kk)=-g(\xx,\n_{\kk}\kk)-d_\xx(g(\kk,\kk))/2=0$,
%where $d_\xx$ denotes the directional derivative for $\xx\in\Gamma(\HH)$. The second holds %similarly since
%$g(\n_{\kk} \xx-\n_{\xx}\kk,\kk)=-g(\xx,\n_{\kk}\kk)-g(\n_{\xx}\kk,\kk)=-(\mathcal{L}_\kk %g)(\xx,\kk)$.
%Note here that a strictly pre-geodesic vector field of constant length is necessarily null, %essentially since
%$0=d_\kk (g(\kk,\kk))/2=g(\n_\kk\kk,\kk)=\al g(\kk,\kk)$.

\vspace{.1in}
Next, condition $g([\kk_+,\cdot],\kk_-)=0$ holds on $\HH$ if the following two conditions
{\em both} hold:
\be\lb{near-grad}
\begin{aligned}
\text{i)  }&\text{$\kk_- = \ell\n\ta$, where $\n\ell\in\Gamma(\VV)$, and}\\
\text{ii) }&\text{$\n(g(\kk_+,\kk_-))\in\Gamma(\VV)$.}
\end{aligned}
\end{equation}
Here $\ta$, $\ell$ are smooth functions.
%and \Ref{near-grad}i) will sometimes be called
%the near-gradient condition.
%If, say, $\kk_-$ is a gradient vector field of constant length, it is geodesic, and additionally
%it is curl-free. If additionally
%\be\lb{mixed}
%\text{$g(\kk_+,\kk_-)$ is constant,}
%\end{equation}
%This follows since if $\xx\in\Gamma(\HH)$, then (ii) and (i) imply
%$g([\kk_+,\xx],\kk_-)=g(\n_{\kk_+}\xx-\n_\xx\kk_+,\kk_-)=-g(\xx,\n_{\kk_+}\kk_-)+g(\kk_+,\n_\xx\kk_-)
%=-g(\xx,\n_{\kk_+}(\ell\n\ta))+g(\kk_+,\n_\xx(\ell\n\ta))=\ell(\n^2\ta(\kk_+,\xx)-\n^2\ta(\xx,\kk_+))=0$,
%because i) implies $\ell$ is constant along vector fields lying in $\HH$.
Note that this will also follow if i) of \Ref{near-grad} is replaced by $g(\kk_+,\kk_-)=0$.

%=\mathrm{curl}\,\kk_-(\xx,\kk_+)=0$.
%Thus under these assumptions, together with a) or b) above,
%three of conditions i) in \Ref{Nij}, hold, so that by Theorem \ref{integ} and Proposition \ref{int-nec},
%integrability of the Nijenhuis tensor is equivalent to the fourth condition in (i) together with (ii).

\vspace{.1in}
If \Ref{near-grad}i) holds, the condition $g([\kk_-,\cdot ],\kk_+)=0$ on $\HH$
%We do not consider special instances where the fourth condition,
%namely $g([\kk_-,\cdot ],\kk_+)=0$, holds on $\HH$. However, if \Ref{near-grad}i)
%holds, it
can be translated into the form
\be\lb{semRm}
\text{$g(\n_{\kk_+}\kk_-+\n_{\kk_-}\kk_+,\cdot)=0$ on $\HH$.}
\end{equation}
%since
%$0=g([\kk_-,\xx],\kk_+)=g(\n_{\kk_-}\xx-\n_\xx\kk_-,\kk_+)=-g(\xx,\n_{\kk_-}\kk_++\n_{\kk_+}\kk_-)$,
%the last equalty holding because $g(\n_\xx\kk_-,\kk_+)=\ell g(\n_\xx\n\ta,\kk_+)=\ell %g(\n_{\kk_+}\n\ta,\xx)=g(\n_{\kk_+}\kk_-,\xx)$.
When both $\kk_+$, $\kk_-$ are pre-geodesic, and $\HH$ is integrable, condition \Ref{semRm} guarantees that $\VV$ is the horizontal distribution for a semi-Riemannian submersion. But integrability of $\HH$ will never, in fact, occur in the circumstances we will be considering later.
%As condition \eqref{semRm}, unlike the geodesic condition, mixes $\kk_+$ and $\kk_-$,
%we will describe it by the phrase
%\be\lb{sm-Rm}
%\text{the distribution $\VV$ satisfies the mixed condition.}
%\end{equation}
%Note finally that \Ref{mixed} and \Ref{sm-Rm} still yield via similar calculations, these last two conditions
%in (i) of Theorem \ref{integ}, if $\kk_-$ has the form
%\be\lb{near-grad}
%\text{$\kk_- = \ell\n\ta$, where $\ell$ is a function constant in directions tangent to $\HH$}.
%\end{equation}
\end{remark}

\subsection{Integrability for split-adjoint admissible almost complex structures}\lb{slf-adj}
An admissible almost complex structure $J=\JJ$ will be called {\em split-adjoint}
if $J\big|_\VV$ is $g\big|_\VV$-self-adjoint. We now give a necessary and
sufficient condition for the integrability of such $J$.

%We present here a generalization of Theorem \ref{integ}.
%In that theorem we only gave a sufficient condition for integrability
%of $J$. It seems too unwieldy to write down a condition that is both necessary
%and sufficient, so equivalent to the vanishing of the Nijenhuis tensor
%in the most general circumstances considered in this paper. Instead we
%consider such an equivalent condition in a special case, which
%involves an assumption on the relation between
%$g\big|_\VV$ and $J\big|_\VV$. Recall that in the Lorentzian case,
%the usual requirement that the metric is hermitian, is
%not available, except on $\HH$. Instead we require that
%$J\big|_\VV$ is self-adjoint. This gives a new and very workable
%set-up for examining integrability.

\begin{thm}
\label{thm:KerrNUT}
%Let $(M,g)$ be semi-Riemmanian with vector fields $\kk_\pm$ satisfying
%\Ref{nnsing} and \Ref{space}, and $J$ the associated almost
%complex structure of Section \ref{cx-str}.
Let $J=\JJ$ be a split-adjoint admissible almost complex structure.
Then $J$ is integrable if and only if the following
conditions hold:
\[
\begin{aligned}
\mathrm{i})\  &g([\kk_-,J\xx],\kk_+)-g([\kk_+,J\xx],\kk_-)-g([\kk_+,\xx],\kk_+)-g([\kk_-,\xx],\kk_-)=0,\\
\text{$\mathrm{ii})$\ }&\text{$J\n^o\kk_+=\n^oJ\kk_+$ on $\HH$},
\end{aligned}
\]
where $\xx$ in i) is any vector field lying in $\HH$.
\end{thm}
\begin{proof}
Recall from Theorem~\ref{integ} that it is enough to analyze the vanishing of
$N(\kk_+,\xx_+)=[\kk_-,   J\xx_+] - J[\kk_+,  J\xx_+] - [\kk_+, \xx_+] - J[\kk_-,  \xx_+] $,
where $\xx_+$, $\xx_-\!\!=\!\!J\xx_+$ is a local oriented orthonormal frame for $\HH$. Taking the inner product of
this expression with $\xx_\pm$ has been carried out in Theorem \ref{integ}, and
led to the shear-condition. This remains unchanged. Condition i) is obtained by first
taking the inner product of $N(\kk_+,\xx_+)$ with $\kk_\pm$, and then employing the
self-adjointness of $J\big|_\VV$, to shift $J$ in the middle two terms in $N(\kk_+,\xx_+)$
from the Lie bracket to the other vector field in the metric expression. Note finally that condition i) is tensorial on $\HH$.
\end{proof}
\begin{remark}\lb{nonttrd}
$J\big|_\VV$ is self-adjoint if and only if $\kk_{\pm}$ have lengths of
opposite signs. This makes Theorem \ref{thm:KerrNUT}
significant in two important cases which appear in our examples:
if both $\kk_{\pm}$ are null, or if they have
nonzero lengths of opposite signs.
Flaherty~\cite{flaherty}, already mentioned in the introduction,
has given a Newman-Penrose version of Theorem~\ref{thm:KerrNUT}
for a special case of the first case, namely if
$\kk_\pm$ are part of a null tetrad. Examples of the first case
appear in Sections~\ref{sec:planewave} and \ref{Kerr}, while those
of the second case are in Sections~\ref{skr} and \ref{sec:shrfl}, in
which $\kk_+$, $\kk_-$ are additionally orthogonal.
\end{remark}
In the following corollary $M$, $g$, $\kk_\pm$, $\VV$, $\HH$ are as in subsection~\ref{defJ}
with \Ref{nnsing}, \Ref{space} holding,
and we only list additional assumptions.
\begin{cor}
Let $\kk_\pm$ be null shear-free %geodesic
vector fields and $f_1$, $f_2$ nowhere
vanishing smooth functions on $M$.
%giving rise to $J$
Assume %$\kk_\pm$, $\HH$ and $\JJ$ satisfy \Ref{Nij}.
$J:=\JJ$ is an admissible complex structure.
Then the admissible almost complex structure $\tilde{J}:=J_{g,f_1\kk_+,f_2\kk_-}$
%ing the assumptions in Theorem \ref{integ},
%then  the vector fields $f_1\kk_+$, $f_2\kk_-$ define, via the procedure described in subsection %\ref{defJ},
%an integrable almost complex structure $\tilde{J}$,
is integrable if and only if $\n(f_1/f_2)\in\Gamma(\VV)$.
\end{cor}
We omit the proof. Note that $\tilde{J}$ is indeed an admissible almost complex structure,
since replacing $\kk_\pm$ by $f_1\kk_+$, $f_2\kk_-$ does not alter  $\VV$, $\HH$,
while it changes $G$ of \Ref{G} by the nowhere vanishing multiple $(f_1f_2)^2$.
%\begin{proof}
%As $f_1\kk_+, f_2\kk_-$ are null, $\tilde{J}$ is split-adjoint and we employ %Theorem~\ref{thm:KerrNUT}. Its condition ii) is verified since $\n^o(f_1\kk_+)=f_1\n^o\kk_+=0$ %and similarly $\n^o(f_2\kk_-)=0$.
%%First, since the shear operator multiplies by a factor when the vector field does, the shear %%operators of $f_1\kk_+$, $f_2\kk_-$ are still zero.
%Next, for $\xx\in\Gamma(\HH)$,
%\[
%g([f_1\kk_+,\xx],f_1\kk_+)=f_1^2g([\kk_+,\xx],\kk_+)-f_1(d_\xx f_1)g(\kk_+,\kk_+)=0
%\]
%because $\kk_+$ satisfies i) of Theorem~\ref{integ} and is null. Similarly, %$g([f_2\kk_-,\xx],f_2\kk_-)=0$.
%Finally,
%\[
%\begin{aligned}
%g([f_1\kk_+,\xx],f_2\kk_-)=&f_1f_2g([\kk_+,\xx],\kk_-)-(d_\xx f_1)f_2g(\kk_+,\kk_-)\\
%g([f_2\kk_-,\xx],f_1\kk_+)=&f_1f_2g([\kk_-,\xx],\kk_+)-(d_\xx f_2)f_1g(\kk_-,\kk_+)
%\end{aligned}
%\]
%so that, via i) of Theorem \ref{integ} again,
%\[
%g([f_1\kk_+,\xx],f_2\kk_-)=g([f_2\kk_-,\xx],f_1\kk_+)
%\]
%if and only if $(d_\xx f_1)f_2=(d_\xx f_2)f_1$, i.e. if and only if $d_\xx(f_1/f_2)=0$ for
%any $\xx\in\Gamma(\HH)$, which is equivalent to $\n(f_1/f_2)\in\Gamma(\VV)$.
%Thus we see that the last condition is equivalent to
%i) of Theorem \ref{thm:KerrNUT} holding with $\kk_+$, $\kk_-$ replaced by $f_1\kk_+$,
%$f_2\kk_-$, respectively,  and $J$ replaced by $\tilde{J}$.
%%As we have already shown ii) of that theorem also holds,
%%as a result of the latter
%%two vector fields being shear-free,
%Hence $\tilde{J}$ is integrable.
%\end{proof}

%--------------------------------
\subsection{Commuting $\kk_\pm$}\lb{holom}
Some further properties of the vector fields $\kk_+$, $J\kk_+$ follow in
the presence of integrability.
\begin{prop}\lb{holo}
%Let $\kk_+$, $\kk_-$ be vector fields satisfying all assumptions in
%Theorem~\ref{integ} for the associated complex structure $J$ described in subsection~\ref{defJ}.
Let $J=\JJ$ be an admissible complex structure.
Then $\kk_+$, $\kk_-$ are holomorphic vector fields with respect to $J$
if and only if they commute and are shear-free.
\end{prop}
\begin{proof}
We examine the Lie derivative formula $(\mathbf{L}_{\kk_\pm}J)a=[\kk_\pm,Ja]-J[\kk_\pm,a]$,
for $a$ taking values in the frame used in the
proof of Theorem \ref{integ}. For $a=\kk_+$ or $a=J\kk_+$, its vanishing will occur if and only if $\kk_\pm$ commute. For $a\in\Gamma(\HH)$, note first that $g((\mathbf{L}_{\kk_\pm}J)a,\kk_\pm)=0$
by \Ref{Nij}i), so \Ref{nnsing1} implies that the projection of $(\mathbf{L}_{\kk_\pm}J)a$
to $\VV$ vanishes. Next, one sees that
$N(\kk_+,\xx_+)=(\mathbf{L}_{\kk_-}J)\xx_++(\mathbf{L}_{\kk_+}J)\xx_-$.
Thus repeating the calculation after \Ref{Nijen} for each summand separately,
we have, for example, $g((\mathbf{L}_{\kk_+}J)\xx_-,\xx_\pm)=0$ if and only if
the shear coefficients of $\kk_+$ vanish, or equivalently the projection
of $(\mathbf{L}_{\kk_+}J)\xx_-$ to $\HH$ vanishes. Since
$J(\mathbf{L}_{\kk_+}J)\xx_-=(\mathbf{L}_{\kk_+}J)\xx_+$, this is also equivalent
to the statement that the restriction of $\mathbf{L}_{\kk_+}J$
to $\HH$ vanishes. Similar considerations lead to the vanishing of the
restriction of $\mathbf{L}_{\kk_-}J$ to $\HH$.
\end{proof}
Note that in our later examples, at times $\kk_+$ and $J\kk_+$ will not commute,
and in one case they will not be shear-free.

%\vspace{.1in}
%If integrability of $J$ is not given, but $\kk_\pm$ commute and are shear-free,
%the above argument shows that both are $J$-holomorphic and $J$ is integrable.
%Similarly if one of them is $J$-holomorphic and both are shear-free, then
%both commute, the second one is thus also $J$-holomorphic and again $J$ is integrable.
%This provides another approach to issues pertaining to a theorem in \cite{calped}.

%\vspace{.1in}
%Since, for our usual frame, $g(\n_{\xx_+}\xx_-+\n_{\xx_-}\xx_+,\kk_-)
%=-g(\xx_-,\n_{\xx_+}\kk_-)-g(\xx_+,\n_{\xx_-}\kk_-)=-2\sigma_2^{\kk_-}$,
%and similarly for $\kk_+$ replacing $\kk_-$, we see that in the shear-free case,
%$\HH$, which is locally spanned by $\{\xx_+,\xx_-\}$, also satisfies (in a local sense)
%a mixed condition.

%\vspace{.1in}
%Also note that in the if $\kk_\pm$ are pre/geodesic and the mixed condition \Ref{semRm}
%holds, the condition that $\kk_+$, $\kk_-$ commute implies that $\VV$ is totally geodesic.
%This follows since then $\n_{\kk_+}\kk_-=\n_{\kk_-}\kk_+$, so the
%mixed condition implies that $\n_{\kk_+}\kk_-$, $\n_{\kk_-}\kk_+$
%lie in $\VV$, and the same holds for $\n_{\kk_+}\kk_+$ and $\n_{\kk_-}\kk_-$.

%----------------------------------------------------------
\section{K\"ahler metrics induced by admissible metrics}
\label{sec:prop0}

This section gives the construction of K\"ahler metrics on
semi-Riemannian $4$-manifolds we call admissible, which, in particular,
possess an admissible complex structure.
%A special
%case in which the K\"ahler metric is a cone metric over
%a Sasaki manifold is also presented.

\subsection{Admissible manifolds and metrics}
We first give the definition of manifold and metric admissibility.
Recall in particular that a complex structure is admissible if
\Ref{nnsing}, \Ref{space} and \Ref{Nij} hold. In this section and
in others we will switch notations to $\kk_+\!\!:=\!\kk$, $\kk_-\!\!:=\!\tT$,
so that $\VV:=\mathrm{span}(\kk,\tT)$.
\begin{defn}\lb{adms}
An oriented semi-Riemannian manifold $(M,g)$ is called {\em admissible} if
it is equipped with two vector fields $\kk$, $\tT$
%, with $\VV:=\mathrm{span}(\kk,\tT)$,
for which
%$\HH:=\VV^\perp:=\mathrm{span}(\kk,\tT)^\perp$, for which
\[
\begin{aligned}
%i]&\text{\  conditions \Ref{nnsing}, \Ref{space} hold,}\\
%ii]&\text{\ conditions \Ref{Nij} hold with $J=J_{g,\kk,\tT}$}\\
%&\text{\ the associated admissible complex structure,}\\
i]&\text{\ $J=J_{g,\kk,\tT}$ is an admissible complex structure,}\\
ii]&\text{\ $\tT=\ell\n\ta$, for $C^\infty$ functions $\ta$, $\ell$,}\\
%with $\n\ell\in\Gamma(\VV)$,}\\
iii]&\text{\  %$\n\ell$,
$\n(g(\kk,\tT))$, $\n(g(\kk,\kk))\in\Gamma(\VV)$.}
\end{aligned}
\]
The metric $g$ is also called admissible.
%More specifically, if additionally
%$\kk$ is geodesic of constant length, or null and pre-geodesic, or Killing,
%then $g$ is called, respectively, geodesic-admissible or pre-geodesic-admissible
%or Killing-admissible.
\end{defn}
We will see in later sections many examples of admissible manifolds.

%----------------------------------
\subsection{The K\"ahler metrics}
We first give certain families of $\J$-invariant symplectic forms
on admissible manifolds. Some of these forms, on Lorentzian manifolds
with $\kk$ null and geodesic, were first considered in \cite{AA}
in analogy with well-known contact forms on $3$-manifolds.

\vskip 6pt
Let $(M,g)$ be an admissible semi-Riemannian $4$-manifold (Definition \ref{adms}),
with vector fields $\kk$, $\tT$.
We consider $2$-forms on $M$ of the form
\be\lb{symp-expl}
\omega:=d(f(\ta)\kk^\flat),
\end{equation}
where $\kk^\flat=g(\kk,\cdot)$, and $f$ is a smooth function defined on the range of $\ta$,
for $\ta$ as in Definition~\ref{adms}ii].
Among the different choices of the ``parameter function" $f$, our interest will lie
mainly in the case where $f$ is affine in $\ta$, or else is $e^\ta$.
As part of the following theorem, we will shortly show that the forms $\om$
are $\J$-invariant. It is thus natural to consider the tensor
\be\lb{kler}
\gK:=\om(\cdot, \J\cdot)=d(f(\ta)\kk^\flat)(\cdot, \J\cdot)
\end{equation}
as a candidate for a K\"ahler metric.
%\begin{remark}\lb{iota}
%Recall that $\ii=\ii^\kk$ is, up to sign, the twist function of $\kk$.
%Throughout the rest of the paper we {\em fix the sign of $\ii$}
%on an oriented $4$-manifold with an admissible almost complex structure
%by the convention that $\ii$ is always computed as in
%\Ref{eqn:twist0} with respect to an {\em oriented orthonormal frame}, oriented to
%agree with the induced orientation of $\HH$.
%This convention will also
%be employed on the (non-admissible) $4$-manifolds of Section~\ref{Kerr}.
%\end{remark}
With $G$ as in \Ref{G}, we have
\begin{thm}\lb{ad-gen}
Let $(M,g)$ be an admissible $4$-manifold.
Then $\gK$, given in \Ref{kler}, is a K\"ahler metric on the
region of $M$ where
\be\lb{ineq}
\text{$f\ii<0$ and $f'G/\ell-f\,d\kk^\flat(\kk,\tT)<0$.}
%\text{$f\ii<0$ and $f'G/\ell-f(g(\n_\kk\kk,\tT)-d_\tT (g(\kk,\kk))/2)<0$.}
\end{equation}
\end{thm}
\begin{proof}
We first check that $\gK$ is $\J$-invariant. It is enough to show
the same for $\om$, and we do so by evaluating it on
$\kk$, $\tT$, $\xx$, where $\xx\in\Gamma(\HH)$.
As there is clearly nothing to check
for a pair of the form $a$, $\J a$, we need not consider the
case where both vectors lie in either of the $J$-invariant distributions
$\VV$, $\HH$. For the case
where one  vector field is in $\VV$ and the other in $\HH$,
we calculate
$\om(\tT,\xx)=d(f\kk^\flat)(\tT,\xx)=fd\kk^\flat(\tT,\xx)
=f[g(\n_\tT\kk,\xx)-g(\n_{\xx}\kk,\tT)]
=-f[g(\kk,\n_\tT\xx)-g(\kk,\n_{\xx}\tT)]=-fg(\kk,[\tT,\xx])=0$,
where here $\n(g(\kk,\tT))\in\Gamma(\VV)$ and \Ref{Nij}i) were applied.
An analogous calculation shows $\om(\kk,\xx)=g(\kk,[\kk,\xx])=0$
using $\n(g(\kk,\kk))\in\Gamma(\VV)$.
%$\om(\kk,\xx')=d(f\kk^\flat)(\kk,\xx')=fd\kk^\flat(\kk,\xx')
%=f[g(\n_\kk\kk,\xx')-g(\n_{\xx'}\kk,\kk)]
%=-f[g(\kk,\n_\kk\xx')+g(\n_{\xx'}\kk,\kk)]
%=-2fg(\kk,\n_{\xx'}\kk)=-f d_{x'}(g(\kk,\kk))=0$, where we have
%successively applied $d\ta(\xx')=g(\tT/\ell,\xx')=0$, $\kk^\flat(\xx')=0$,
%$g(\kk,[\kk,\xx'])=0$ and $\n(g(\kk,\kk))\in\Gamma(\VV)$. On
%the other hand $\om(\tT,\xx')=d(f\kk^\flat)(\tT,\xx')=fd\kk^\flat(\tT,\xx')
%=f[g(\n_\tT\kk,\xx')-g(\n_{\xx'}\kk,\tT)]
%=-f[g(\kk,\n_\tT\xx')-g(\kk,\n_{\xx'}\tT)]=-fg(\kk,[\tT,\xx'])=0$,
%where here $\n(g(\kk,\tT))\in\Gamma(\VV)$ was applied.
As both calculations yield zero, $\om$ and hence $\gK$ are $\J$-invariant.

Now $\om$ is clearly exact, and $\J$ is integrable, % by Theorem \ref{integ},
so it remains to compute the region where $\J$ is $\om$-tame, yielding positive
definiteness of $\gK$. It is clearly enough to check when
\be\lb{pos}
\gK(a,a)=\om(a,Ja)>0
\end{equation}
for all vector fields $a$ of a $\gK$-orthogonal frame. We choose a
frame $\kk$, $\tT$, $\xx$, $\yy=\J\xx$ with the latter two (possibly local)
vector fields forming an oriented $g$-orthonormal frame for $\HH$. We need only
compute $\om(\xx,\yy)$ and $\om(\kk,\tT)$.
The first is
$\om(\xx,\yy)=fd\kk^\flat(\xx,\yy)=f(g(\n_{\!\xx}\kk,\yy)-g(\n_{\!\yy}\kk,\xx))=-f\ii$,
by equation \Ref{eqn:twist0}.
For the second, we write
\be\lb{decom}
\om=f'd\ta\we \kk^\flat+fd\kk^\flat.
\end{equation}
Then
$f'(d\ta\we\kk^\flat)(\kk,\tT)=-f'[g(\kk,\tT)^2-g(\tT,\tT)g(\kk,\kk)]/\ell=-f'G/\ell$,
so that $\om(\kk,\tT)=-f'G/\ell+fd\kk^\flat(\kk,\tT)$. These expressions yield
the inequalities in the theorem defining the region where $\gK$ is K\"ahler.
\end{proof}
We will often describe a K\"ahler metric provided by the above theorem
as \emph{induced by an admissible metric}.

\vspace{.05in}
With $\al$ below given in \Ref{pr-gd}, we have
\begin{remark}\lb{LIK}
The second inequality in \Ref{ineq}, for the region where $\gK$ in \Ref{kler} represents
a K\"ahler metric, becomes
\[
\begin{aligned}
%$f\ii<0$,
&\text{$f'G/\ell<0$ if $\kk$ is geodesic of constant length, or}\\[3pt]
&\text{$f'G/\ell-f\al g(\kk,\tT)<0$ if $\kk$ is null and strictly pre-geodesic, or}\\[3pt]
&\text{$f'G/\ell+fd_\tT(g(\kk,\kk))<0$ if $\kk$ is Killing.}
\end{aligned}
\]
%\end{cor}
The first of these inequalities
can be simplified to $f'/\ell>0$ in the Lorentzian case,
using \Ref{G}.
%\begin{proof}

In fact, %in the second term in \Ref{ineq},
$d\kk^\flat(\kk,\tT)=g(\n_\kk\kk,\tT)-g(\n_\tT\kk,\kk)=g(\n_\kk\kk,\tT)-d_\tT(g(\kk,\kk))/2$
vanishes when $\kk$ is geodesic with constant length, equals $\al g(\kk,\tT)$ when $\kk$ is a null
and strictly pre-geodesic, and equals $-g(\kk,\n_\tT\kk)-d_\tT(g(\kk,\kk))/2=-d_\tT(g(\kk,\kk))$
if $\kk$ is Killing.
%\end{proof}
\end{remark}
%In Section~\ref{sec:shrfl} an example of an
%admissible metric that is not geodesic-, pre-geodesic- or Killing-admissible
%will be given.
\begin{remark}\lb{doubleKJ}
Let $J_+:=J$, and denote by $J_-$ the almost complex structure
of Remark~\ref{doubleJ}, respecting the opposite orientation.
%As the proof of Theorem~\ref{ad-gen} shows that $\HH$ and $\VV$
%are also $\gK$-orthogonal, it follows that $\gK$ is almost hermitian
%also with respect to $J_-$. %By Remark~\ref{doubleJ},
It is easy to show that if $\kk_\pm$ are shear-free,
%$\gK$ is in fact hermitian with respect to these
%two complex structures, which respect opposite orientations. In other words
{\em $\gK$ is ambihermitian}, in the sense of \cite{acg}.
In some cases, $\gK$ is also {\em ambiK\"ahler}
%, meaning that
%there exists another metric in its conformal class which is K\"ahler
%with respect to this second complex structure.
%This happens,
(e.g. %for example
when $\gK$ is an SKR metric as in Section~\ref{skr}).
%A separate observation is that $\om(\cdot,J_-\cdot)$ is a metric,
%not in the same conformal class, which differs from $\gK$ by having
%a different sign on its restriction to $\HH$. This can be seen
%directly, and also follows since $\iota$ switches sign when
%the ordered basis of $\HH$ is flipped in the passage from $J_+$ to $J_-$.
%As a result this metric has signature $(2,2)$ in the domain of $\gK$,
%while it is K\"ahler in the region where the expression
%defining $\gK$ has signature $(2,2)$.
\end{remark}

%----------------------------------------------------
\subsection{$\gK$ values on a frame}
\label{sec:inher}

We record a number of basic facts about $\gK$ inferred from Theorem~\ref{ad-gen}.
\begin{lemma}\lb{basic}
For a K\"ahler metric induced from an admissible semi-Riemannian metric,
\[
\gK(\VV,\HH)=0,\qquad \gK(\kk,\tT)=0,\qquad \gK(\kk,\kk)=\gK(\tT,\tT),\qquad
\gK\big|_\HH=-f\ii g\big|_\HH.
\]
\end{lemma}
\begin{proof}
The first equality was shown in the proof of Theorem~\ref{ad-gen}.
The second and third are obvious as $\om$ is a $2$-form. The fourth follows
by comparing, with the help of \Ref{decom}, $\gK\big|_\HH(a,b)=fd\kk^\flat(a,Jb)=f(g(\n_a\kk,Jb)-g(\n_{Jb}\kk,a))$
and $-f\ii g\big|_\HH(a,b)$ using \Ref{eqn:twist0}, where $a$, $b$ are taken
from an oriented orthonormal frame $\xx$, $\yy\!\!=\!\!J\xx$ for $\HH$.
\end{proof}
Note that $\gK(\kk,\kk)=\gK(\tT,\tT)=fd\kk^\flat(\kk,\tT)-f'G/\ell$ (see
Theorem~\ref{ad-gen}), which reduces to minus one of the expressions in
the inequalities of Remark~\ref{LIK}, in the case where $\kk$ satisfies one
of the corresponding assumptions given there. This formula will also be used
later in the easily verifiable form
\be\lb{alt}
\begin{aligned}
&\gK(\kk,\kk)=\gK(\tT,\tT)=-f(g([\kk,\tT],\kk)+d_\tT(g(\kk,\kk))-f'G/\ell,\\
&\text{whenever $g(\kk,\tT)$ is constant.}
\end{aligned}
\end{equation}
%\begin{remark}\lb{kt}
%Note that as $\om$ is a $2$-form, we also have
%\[
%\gK(\kk,\kk)=\gK(\tT,\tT).
%\]
%It follows that in one of our standard frames of the form $\{\kk,\tT,\xx,\yy\}$,
%the matrix of $\gK$ is diagonal, with at most two distinct diagonal values at each
%point.
%\end{remark}

\subsection{Additional K\"ahler metrics and functions with a gradient in $\VV$}

If $g$ is an admissible metric with distinguished {\em shear-free} vector fields $\kk$, $\tT$,
then $g$ is also admissible relative to the pairs  $a\kk$, $\tT$ and
$\kk$, $a\tT$, whenever $a$ is a function with $\n a\in\Gamma(\VV)$. This can
be shown easily by verifying i]-iii] of Definition~\ref{adms} for these pairs.
Thus $g$ induces two new families of K\"ahler metrics for each such function $a$.
The K\"ahler metrics relative to the pair $\kk$, $a\tT$ will have the same symplectic
form as the original $\gK$ but a different complex structure. Note, however,
that the function $a$ must, in fact, be functionally dependent on $\tau$:
\begin{prop}\lb{ver-grad}
On the domain of any associated K\"ahler metric in an admissible manifold,
any function $a$ with gradient in $\VV$ must be
locally a function of $\tau$ of Definition~\ref{adms}.
\end{prop}
\begin{proof}
If $\n a=\alpha \kk+\beta \tT$ then $da=\alpha \kk^\flat+\beta \tT^\flat$,
so applying the exterior derivative on both sides gives
$0=d\alpha\wedge \kk^\flat+\alpha d\kk^\flat+d\beta\wedge \tT^\flat+\beta d\tT^\flat$.
Applying this to the usual oriented orthonormal frame $\xx$, $\yy$ of $\HH$ gives
$0=\alpha d\kk^\flat(\xx,\yy)+\beta d\tT^\flat(\xx,\yy)$.
Using $d\phi(a,b)=d_a(\phi(b))-d_b(\phi(a))-\phi([a,b])$
for $\phi=\kk^\flat$ and $\phi=\tT^\flat$ along the orthogonality
of $\VV$ and $\HH$ yields
$0=-\alpha \kk^\flat([\xx,\yy])-\beta \tT^\flat([\xx,\yy])=-\alpha\iota^\kk-\beta\iota^\tT$.
But the twist function $\iota^\tT=0$ since
$\iota^\tT=g(\ell\n\ta,\n_\xx\yy-\n_\yy\xx)=-\ell(g(\n_\xx\n\ta,\yy)-g(\n_\yy\n\ta,\xx))=0$
as the Hessian is symmetric. Thus, as $\iota^\kk$ is nonzero on the domain of an
associated K\"ahler metric, $\alpha=0$. So $da=\beta\tT^\flat=\beta\ell d\ta$, hence
$a$ is locally a function of $\ta$.
\end{proof}
In particular, iii] of Definition~\ref{adms} implies that for an
admissible metric, on the domain of an induced K\"ahler metric,
$g(\kk,\tT)$ and $g(\kk,\kk)$ are always functions of $\ta$.

%--------------------------------------------------------------
\section{First-order properties of the induced K\"ahler metric}

In this section we will show that some properties of
special admissible metrics are also shared by their induced
K\"ahler metric, including admissibility itself. We will also
examine other metric characteristics, such as
completeness.

\subsection{Relations between the admissible and induced metrics}
\label{sec:prop}

We will occasionally employ the notation $\tc:=\ta-c$, for a constant $c$.
\begin{prop}\lb{geo-geo}
Let $g$ be an admissible semi-Riemannian metric with $\kk$, $\tT$ geodesic vector fields
of constant length, $g(\kk,\tT)$ constant and $\ell=1$. If $\gK$ is the induced K\"ahler
metric
%according to Theorem \ref{ad-gen}
with $f(\ta)=\tc$, then the associated vector fields $\kk$, $\tT$ are geodesic and of constant length also with respect to $\gK$.
\end{prop}

\begin{proof}
First, $\kk$ has constant $\gK$-norm as $\kk$ is geodesic of
constant length, $f'=\ell=1$ and $G$ is constant (see Remark~\ref{LIK}
and the end of subsection~\ref{sec:inher}). By Lemma~\ref{basic}, $\tT$
also has constant $\gK$-norm.
Let $\xx$, $\yy$ be a local orthonormal frame
of $g\big|_H$, and $\nK$ denote the Levi-Civita connection of $\gK$.
By the Koszul formula, the constant length of $\kk$ implies
$2\gK(\nK_{\!\kk}\kk,\xx)=-2\gK(\kk,[\kk,\xx])=0$, the last equality following  as $[\kk,\xx]$
lies in $\HH$, since admissible metrics satisfy the hypotheses of Theorem \ref{integ}.
The same holds for $\xx$ replaced by $\yy$, and
clearly also $\gK(\nK_{\!\kk}\kk,\kk)=d_\kk(\gK(\kk,\kk))/2=0$.

%Now write the $\VV$-component of $[\kk,\tT]$ as
%$[\kk,\tT]^\VV=\alpha\kk+\beta\tT$ for some coefficient functions $\alpha$, $\beta$.
The Koszul formula for $g$ also gives
\[
\begin{aligned}
&0=-g(\n_{\!\kk}\kk,\tT)=g(\kk,[\kk,\tT]),\\%=\alpha g(\kk,\kk)+\beta g(\kk,\tT)\\
&0=-g(\n_{\!\tT}\tT,\kk)=g(\tT,[\tT,\kk])%=-\alpha g(\tT,\kk)-\beta g(\tT,\tT)
\end{aligned}
\]
%As the matrix $A$ of $g\big|_\VV$ is nonsingular (see \Ref{nnsing}), we have $\alpha=\beta=0$
so $[\kk,\tT]$ is tangent to $\HH$. Applying this, we see by
invoking the Koszul formula once more, this time
for $\gK$, that, as $\gK\big|_\VV$ is constant on $\kk, \tT$,
$\gK(\nK_{\!\kk}\kk,\tT)=-\gK(\kk,[\kk,\tT])=0$ by the first relation in Lemma~\ref{basic}. Thus $\gK$ vanishes on
any pair consisting of $\nK_{\!\kk}\kk$ and any vector field in the $\gK$-orthogonal frame $\{\kk,\tT,\xx,\yy\}$,
so %using again that $A$ is nonsingular,
$\nK_{\!\kk}\kk=0$. A similar argument shows
that %$\tT$ has constant length and
$\nK_{\!\tT}\tT=0$.
\end{proof}

A corresponding result when $\kk$ is Killing is the following.
\begin{prop}\lb{Killing}
Let $g$ be an admissible semi-Riemannian metric with $\kk$ Killing,
%$\n(g(\kk,\kk))\in\Gamma(\VV)$,
$\ell$, $\ii$ functions of $\ta$, $g(\kk,\tT)=0$ and $[\kk,\tT]=0$. Then for any induced
K\"ahler metric $\gK$, $\kk$ is also $\gK$-Killing.
\end{prop}
%\begin{proof}
%By \Ref{shr-div} and the fact that $\kk$ is $g$-Killing, we see that the $\gK$-shear of
%$\kk$ vanishes.
We omit the proof, which consists of a systematic verification of the Killing field identity
\be\lb{kill}
d_\kk(\gK(a,b))=\gK([\kk,a],b)+\gK(a,[\kk,b])
\end{equation}
when $a$, $b$ are taken from the usual frame vector fields $\kk,\tT,\xx,\yy\!\!=\!\!J\xx$.
Examples where all these conditions are satisfied will be given in Section \ref{skr}.

%\vspace{.1in}
%Next, we have
%\begin{lemma} If $g$ is an admissible semi-Riemannian metric, with $g(\kk,\tT)=0$ and %$g(\kk,\kk)$
%a function of $\tau$ of \Ref{near-grad}i), then $\tT=\tilde{\ell}\nK\ta$, with the
%$\gK$-gradient of the function $\tilde{\ell}$ lying in $\VV$.
%\end{lemma}
%\begin{proof}
%Let $\xx$, $\yy$ be an orthonormal frame for $g\big|_\HH$.
%Since $\gK(\nK\ta,b)=d\ta(b)=g(\n\ta,b)=g(\tT,b)/\ell$ and $\gK(\tT,b)$
%are both zero for $b=\kk,\xx,\yy$, the only nonzero component of $\nK\ta$
%in the usual frame (a frame which is $\gK$-orthogonal), is a multiple of $\tT$.
%The multiple itself can be found as follows, where we do the calculation only in the
%Killing-admissible case, as the other cases are similar. If $p:=g(\kk,\kk)$ and $q:=g(\tT,\tT)$,
%then the multiple is $\gK(\nK\ta,\tT)/\gK(\tT,\tT)=(q/\ell)/(-f'pq/\ell-f\,d_\tT p)
%=(-f'pq/\ell-fp'\,d_\tT\ta)=(q/\ell)/[-(q/\ell)(fp)']=-1/(fp)'$. But the result is
%a function of $\ta$, whose $\gK$-gradient we know to be parallel to $\tT$, which lies in $\VV$.
%\end{proof}

%Furthermore,
%\begin{prop}
%For a K\"ahler metric induced by an admissible semi-Riemannian metric,
%$\VV$ satisfies the mixed condition with respect to $\gK$.
%\end{prop}
%This follows, in fact, from the Koszul formula which, as $\gK(\kk,\tT)=0$ and $[\kk,\xx]$, %$[\tT,\xx]$
%are in $\HH$ for $\xx$ in $\HH$, gives
%\[
%2\gK(\nK_\kk\tT+\nK_\tT\kk,\xx)=\gK(\xx,[\kk,\tT])+\gK(\xx,[\tT,\kk])=0.
%\]

\vspace{.1in}
Finally, %we show that
the shear condition \Ref{Nij}ii) of Theorem~\ref{integ}
holds also with the shears taken with respect to $\gK$. In fact, this follows
from the following relations between shears with respect to
to an admissible metric $g$ and a K\"ahler metric $\gK$ it induces:
\be\lb{shr-div}
(\nK\!)^o\kk=\n^o\kk, \qquad (\nK\!)^o\tT=\n^o\tT.
%\qquad%\del_\HH^K\kk=\del_\HH\kk-2\sqrt{f\ii^\kk}d_\kk((f\ii^\kk)^{-1/2}),
\end{equation}
%where the superscript $K$ denotes that the quantity in question is computed with respect to $\gK$.
These relations can be verified by direct calculation.
%We prove the first relation only, as the second is similar.
%Let $\{\xx,\yy\}$ be, as usual, a local orthonormal frame for
%$g|_\HH$. From \Ref{eqn:shear2} we have
%Computing $g(\n_{\!\yy}\kk,\yy)$ and $g(\n_{\!\xx}\kk,\xx)$ via the Koszul formula and %subtracting %the resulting expressions for the shear coefficients \Ref{eqn:shear2} of $\kk$
%yields, as all $g$-inner products of pairs from $\kk$, $\xx$, $\yy$ are constant,
%\[
%2\sig^\kk_1=-g(\yy,[\kk,\yy])+g(\xx,[\kk,\xx]).
%\]
%We now compute the corresponding shear coefficient for the K\"ahler metric $\gK$ and the %$\gK$-orthonormal basis
%$\tilde{\xx} := \xx/s$, $\tilde{\yy} := \yy/s$, where $s:=\sqrt{-f\ii^\kk}$.
%The Koszul formula for $\gK$ gives, using the first relation in Lemma \ref{basic},
%\[
%\begin{aligned}
%2\sig^{\scriptscriptstyle{K},\kk}_1&=-\gK(\tilde{\yy},[\kk,\tilde{\yy}])+\gK(\tilde{\xx},[\kk,\tilde{\xx}])\\
%&=-s^{-1}[\gK(\yy,d_\kk(s^{-1})\yy+s^{-1}[\kk,\yy])-[\gK(\xx,d_\kk(s^{-1})\xx+s^{-1}[\kk,\xx])]\\
%&=-s^{-1}d_\kk(s^{-1})[\gK(\yy,\yy)-\gK(\xx,\xx)]-s^{-2}s^2[g(\yy,[\kk,\yy]]-g(\xx,[\kk,\xx])]\\
%&=2\sig^\kk_1,
%\end{aligned}
%\]
%since $\gK\big|_\HH=s^2 g\big|_\HH$ so that %$\gK(\yy,\yy)=s^2g(\yy,\yy)=s^2g(\xx,\xx)=\gK(\xx,\xx)$.
%A similar calculation shows $\sig^{\scriptscriptstyle{K},\kk}_2=\sig^\kk_2$ and thus
%the shear relation for $\kk$ follows.
%\vskip 6pt
Of course, there is a more conceptual argument for the validity of the shear
condition \Ref{Nij}ii) with respect to the K\"ahler metric. Namely, the proof
of Theorem \ref{integ} shows that if conditions \Ref{Nij}i) hold, then the
Nijenhuis tensor vanishes if and only if the shear condition \Ref{Nij}ii) is
satisfied. This vanishing, along with \Ref{Nij}i), are of course metric
independent conditions, so it follows that, assuming \Ref{Nij}i), if
$J$ is integrable, condition \Ref{Nij}ii) must hold for any metric for
which \Ref{nnsing} and \Ref{space} hold.
Now \Ref{space} certainly holds for the Riemannian metric $\gK$,
while \Ref{nnsing} also holds as it is independent of the metric.
%As $\gK\big|_\VV$ is nondegenerate, according to Remark \ref{re-phrs},
%condition \Ref{nnsing} follows automatically from the metric independent
%condition that $\kk$, $\tT$ are linearly independent at every point.

%The mixed condition for $\gK$ can also be explained similarly.
%Generally speaking, the independence of i) and ii) of
%Theorem \ref{integ} from the particular metric employed forms
%the underlying reason for most of the results of this section.
%However, as some of the conditions analyzed in this section
%represent special circumstances in which i) holds, more assumptions
%are sometimes needed, as we have seen, to show that such circumstances
%also occur with $\gK$.

%------------------------------------------------------
\subsection{Repeated admissibility}\lb{repeat}
In the following proposition admissible metrics induce K\"ahler metrics that are also
admissible.
%\begin{lemma} If $g$ is an admissible semi-Riemannian metric with $g(\kk,\tT)=0$ and $g(\kk,\kk)$
%a function of $\tau$ of \Ref{near-grad}i), then $\tT=\tilde{\ell}\nK\ta$, with the
%$\gK$-gradient of the function $\tilde{\ell}$ lying in $\VV$.
%\end{lemma}
%\begin{proof}
%Let $\xx$, $\yy$ be an orthonormal frame for $g\big|_\HH$.
%Since $\gK(\nK\ta,b)=d\ta(b)=g(\n\ta,b)=g(\tT,b)/\ell$ and $\gK(\tT,b)$
%are both zero for $b=\kk,\xx,\yy$, the only nonzero component of $\nK\ta$
%in the usual frame, which is $\gK$-orthogonal, is a multiple of $\tT$.
%The multiple itself can be found as follows. If $p:=g(\kk,\kk)$ and $q:=g(\tT,\tT)$,
%then the multiple is $\gK(\nK\ta,\tT)/\gK(\tT,\tT)=(q/\ell)/(-f'pq/\ell-f\,d_\tT p)
%=(-f'pq/\ell-fp'\,d_\tT\ta)=(q/\ell)/[-(q/\ell)(fp)']=-1/(fp)'$. But the result is
%a function of $\ta$, whose $\gK$-gradient we know to be parallel to $\tT$, which lies in $\VV$.
%\end{proof}
\begin{prop}\lb{repeat2}
Let $g$ be an admissible semi-Riemannian metric on an oriented $4$-manifold $M$
with distinguished vector fields $\kk$, $\tT$.
Assume additionally that $g(\kk,\tT)=0$ and
$\ell$, %$g(\kk,\kk)$,
$g(\tT,\tT)$, $g([\kk,\tT],\kk)$ are all functions of $\ta$.
%with $g(\kk,\kk)$ nowhere vanishing.
For a given $f:\mathrm{Im}\,\ta\to\mathbb{R}$, let $\gK=\gK(g,\kk,\tT,f)$ be
an induced K\"ahler metric on an appropriate region. Then $\{\gK,\kk,\tT\}$
is admissible and satisfies all the conditions mentioned above for $g$. Thus there exists
an infinite sequence of admissible K\"ahler metrics $\gK^{(n)}$, $n\in\mathbb{N}$,
with $\gK^{(1)}:=\gK$, each defined by the quadruple $\{\gK^{(n-1)},\kk,\tT, f$\} on
an appropriate region $U^{(n)}\subset M$ (which may, however, be empty for some $n$).
\end{prop}
\begin{proof}
To show that $\gK$ is admissible, taking into account that it is Riemannian,
that $\gK(\kk,\tT)=0$ and that \Ref{shr-div} holds, it is enough to check
$\nK(\gK(\kk,\kk))\in\Gamma(\VV)$ and $\tT=\ell_K\nK\ta$ for some function $\ell_K$.

For the first of these, our assumptions together with \Ref{alt} and the last
sentence of Section~\ref{sec:prop0} guarantee that $\gK(\kk,\kk)$
is a function of $\ta$. Since $d_\xx\ta=g(\tT/\ell,\xx)=0$ for any $\xx\in\Gamma(\HH)$,
the claim follows.
%First we check the metric dependent admissibility conditions %of Theorem $4$
%for $\gK$,
%together with their stronger version needed in the current theorem.
%Clearly $\gK\big|_\HH$ is positive definite.
%The proof that $\tT=\tilde{\ell}\nK\ta$ with $\nK\tilde{\ell}\in\Gamma(\VV)$
%(in fact we will show that $\tilde{\ell}$ is a function of $\ta$).
%exactly as in Lemma 5.5 in \cite{a-m}, except that we need to verify
%that (just in the admissible case presented there) the ratio
%$\gK(\nK\ta,\tT)/\gK(\tT,\tT)$ is a function of $\ta$.

For the second claim, let $\xx$, $\yy$ be an orthonormal frame for $g\big|_\HH$.
Since $\gK(\nK\ta,b)=d\ta(b)=g(\n\ta,b)=g(\tT,b)/\ell$ and $\gK(\tT,b)$
are both zero for $b=\kk,\xx,\yy$, the only nonzero component of $\nK\ta$
in the usual frame, which is $\gK$-orthogonal, is a multiple $\ell_K$ of $\tT$.
Thus $\gK$ is admissible.

It remains to verify that the other assumptions of the proposition hold for
$\gK$. We already know that $\gK(\kk,\tT)=0$ and that $\gK(\tT,\tT)=\gK(\kk,\kk)$
is a function of $\ta$.
For $\ell_K$ this is also clear since so are $\gK(\nK\ta,\tT)=g(\tT,\tT)/\ell$ and $\gK(\tT,\tT)$.
Finally, expanding in the usual frame $[\kk,\tT]=a\kk+\ldots$,
clearly $g([\kk,\tT],\kk)=ag(\kk,\kk)$, so $a$ is a function of $\ta$
($g(\kk,\kk)$ is nowhere vanishing as $G\ne 0$). Hence so is $\gK([\kk,\tT],\kk)=a\gK(\kk,\kk)$.

%The multiple itself can be found as follows. If $p:=g(\kk,\kk)$ and $q:=g(\tT,\tT)$,
%then the multiple is $\gK(\nK\ta,\tT)/\gK(\tT,\tT)=(q/\ell)/(-f'pq/\ell-f\,d_\tT p)
%=(-f'pq/\ell-fp'\,d_\tT\ta)=(q/\ell)/[-(q/\ell)(fp)']=-1/(fp)'$. But the result is
%a function of $\ta$, whose $\gK$-gradient we know to be parallel to $\tT$, which lies in $\VV$.

%But the calculation there shows the numerator is $g(\tT,\tT)/\ell$
%while the denominator is $-f'G/\ell-f(g([\kk,\tT],\tT)+d_\tT(g(\kk,\kk)))$,
%so this ration and hence $\tilde{\ell}$ are clearly functions of $\ta$, whose $\gK$-gradient is
%a multiple of $\tT$, hence lies in $\VV$. For similar reasons the $\gK$-gradients
%of $\gK(\kk,\kk)=\gK(\tT,\tT)$ (and $\gK(\kk,\tT)=0$) lie in $\VV$, and in fact
%are functions of $\ta$.
%For reasons described at the end of Section 5.1 in \cite{a-m},
%the shear integrability condition holds with respect to $\gK$.
%Finally, expanding in our frame $[\kk,\tT]=a\kk+\ldots$,
%clearly $g([\kk,\tT],\kk)=ag(\kk,\kk)$, so $a$ is a function of $\ta$.
%Hence so is $\gK([\kk,\tT],\kk)=a\gK(\kk,\kk)$.
\end{proof}
One situation in which the domains of these successive K\"ahler metrics are never empty
is described in the following proposition, whose proof is omitted.
\begin{prop}
Under the conditions of Proposition~\ref{repeat2}, if additionally $p:=g(\kk,\kk)>0$,
$fp$ is constant and $fa$ is constant for a nonzero coefficient
$a:=g([\kk,\tT],\kk)/g(\kk,\kk)$, then the domains of $\gK^{(n)}$ satisfy
$U^{(n)}\subset U^{(n+1)}$ for all $n$.
\end{prop}
\subsection{Completeness of geodesic vector fields}\lb{comp}
%We mention here an outcome related to the assumption of completeness of $\gK$.
\begin{prop}
Let $g$ be an admissible semi-Riemannian metric on a manifold $M$, with $\kk$ a
geodesic vector field of constant length. Suppose for some choice of a function
$f(\tau)$, a complete K\"ahler metric $\gK$
is induced %as in Theorem \ref{ad-gen}
on {\em all of} $M$. Then any inextendible integral curve of $\kk$ is defined on the whole real line if and only if $f'G/\ell$ is
bounded below on the curve.
\end{prop}
\begin{proof}
According to \cite[Proposition 3.4]{SC}, a geodesic will be extendible
beyond a finite interval of its parameter domain if and only if in some
complete Riemannian metric the length of its tangents is bounded.
Taking $\gK$ to be this metric, and the geodesic an integral curve of $\kk$,
extendibility will occur %for any finite interval
if and only if $\gK(\kk,\kk)$ is bounded along it. But $\gK(\kk,\kk)=-f'G/\ell$
(see Remark~\ref{LIK}).
\end{proof}
Note that as $\gK(\tT,\tT)=\gK(\kk,\kk)$,  if $\tT$ is also geodesic,
the same holds for its integral curves. If $g$ is Lorentzian or one of
these vector fields is null, the same holds if $f'/\ell$ is bounded
above on the curve, as $G$ is negative (see \Ref{G}).

Examples of complete K\"ahler metrics induced from admissible manifolds
appear in subsection~\ref{complt}.

\subsection{Admissible metrics inducing the same K\"ahler metric}\lb{invert}

%In this short subsection
We briefly outline, without proof, a pair of circumstances in which two
different semi-Riemannian metrics induce the same K\"ahler metric (for a given $f(\ta)$).
%As the proofs involve arguments similar to ones we have already made in
%previous sections and the next one, we will be brief.

%Suppose on a given oriented $4$-manifold $\kk$, $\tT$ are fixed pointwise linearly independent
%vector fields spanning $\VV$, and an almost complex structure $J$ is given by its values on $\VV$
%as in subsection \ref{defJ}, and those on $\HH=\VV^\perp$ stipulated independently.
%Suppose functions $\ell$, $\ta$ on this manifold are also given.
%Formula \Ref{symp-expl} for the symplectic form shows that if
%two semi-Riemannian metrics $g$, $\tg$ whose restriction to $\HH$ is hermitian with
%respect to $J$, satisfy $\kk^\flat=\kk^{\tilde{\flat}}$ and $\tT=\ell\n\ta=\ell\tilde{\n}\ta$,
%they will yield the same symplectic form, and hence K\"ahler metric (for the
%same $f$). But this condition implies
%$g(\kk,\kk)=\tilde{g}(\kk,\kk)$ and $g(\kk,\tT)=\tg(\kk,\tT)$, so that
%the only way $g\big|_\VV$ and $\tg\big|_\VV$ can differ is if
%$g(\tT,\tT)\ne\tg(\tT,\tT)$. And this may occur even with the condition
%above on $\tT$. We do not discuss in this preliminary remark
%whether other admissibility conditions may hold for both metrics.
%Instead we turn to explicit ways to vary $g$ that yield the same K\"ahler metric.

\subsubsection{Varying $g\big|_\HH$}
On a given oriented $4$-manifold fix two pointwise linearly independent
vector fields $\kk$, $\tT$ spanning a distribution $\VV$. Given a semi-Riemannian
admissible metric $g$, consider a biconformal change of the form
%Even if $\kk$, $\tT$ are fixed and the two metrics agree on $\VV$,
%they may still differ on $\HH$ while inducing the same K\"ahler
%metric.
%One situation where this occurs is in a biconformal change of the form
\[
g=g\big|_\VV\oplus g\big|_\HH\to\tg=g\big|_\VV\oplus\beta^2g\big|_\HH
\]
for the usual $\HH$ and a nowhere vanishing function $\beta$.
One can show that $\tg$ is admissible and for fixed $f(\tau)$ one has $\tg_{\scriptscriptstyle{K}}=\gK$.

In Section \ref{skr} admissible metrics %with $\kk$-Killing
will be produced from a special type of K\"ahler metric that
they, in turn, induce. In that setting there
will be a canonical choice for $g\big|_\HH$ in such a biconformal class.

\subsubsection{Varying $g\big|_\VV$}

We now give a complementary change of metric, where $g\big|_\HH$ is fixed
while $g\big|_\VV$ varies.
Let $g$ be an admissible semi-Riemannian metric %with $\kk$ geodesic
%of constant length
satisfying the hypotheses of %Theorem \ref{ad-gen}, and again those of
Proposition \ref{geo-geo}. Let $\bg=g+\ep d\ta^2$ for
small $\ep>0$. The parameter $\ep$ is only introduced so that $g$ and $\bg$
have the same signature. Again one can show $\bg$ is admissible,
and for a given $f(\tau)$ the symplectic forms and admissible almost complex
structures associated with $g$ and $\bg$ coincide, so that $\bg_{\scriptscriptstyle K}=\gK$.
\section{Ricci and scalar curvatures of $\gK$}
\label{Ric-scal}

%Suppose an admissible metric satisfies $g(\kk,\kk)=-g(\tT,\tT)$
%and $g(\tT,\kk)=0$. Due to Lemma \ref{basic},
%similar relations (up to sign) hold for any induced K\"ahler metric
%$\gK$, and this lemma also shows that on $\HH$,
%$\gK$ is a function multiple of $g$. This verifies the claim made in
%the introduction, that after changing the sign of $g(\tT,\tT)$
%(an operation called a metric rotation in the introduction), the resulting
%metric $\tilde{g}$ and $\gK$ become biconformal,
%i.e. their restrictions to $\VV$ and $\HH$ are each a function multiple
%of the other. This implies a fairly complicated relation
%between their connections and curvatures, which we will not pursue
%at this time.
We give here formulas for the Ricci and scalar curvature of an
induced K\"ahler metric $\gK$ under certain assumptions,
%based on their very definition via
%a symplectic form, which easily yield their volume form. These
%formulas hold with extra assumptions,
most importantly that the
metrics conform to a certain bundle structure.

\vspace{.08in}
The Ricci form of a K\"ahler metric is given via
\be\lb{ric}
\rho=-i\partial\bar{\partial}\log (\mu/\nu)=-\fr 12dJd\log (\mu/\nu),
\end{equation}
where $\mu$ is the volume form coefficient in a coordinate system and
$\nu$ is the coefficient for the coordinate volume form in corresponding
complex coordinates (cf. \cite[\S1.4.3]{varo}).
We apply this general formula to the case of the K\"ahler
metric $\gK$ associated to an admissible semi-Riemannian manifold,
computed under certain assumptions detailed below. We give
two formulas, one for the case that $\kk$ is geodesic of constant length,
the other for the case it is Killing. Examples fulfilling these assumptions
will be given in Section \ref{skr}.

\begin{prop}\lb{rc-scal1}
Let $(M,g)$ be admissible, with admissible complex structure $J=\J$.
Assume that $M$ is an open set in the total space of a holomorphic line bundle over a
(two dimensional) K\"ahler surface $(N,h)$ with holomorphic projection map $\pi$, and
$g\big|_\HH=\pi^*h$. Suppose $\kk$, $\tT$ commute and are shear-free,
$\ell$ is a function of $\ta$, $\kk$ is geodesic of constant nonzero length
and $\kk^\flat$ is locally a sum of an exact form and one vanishing on $\VV$.

Let $\om^h=r\, dx\we dy$ be the K\"ahler form of $h$, expressed in
local coordinates where $z=x+iy$ is a holomorphic
coordinate on $N$.
Then for any K\"ahler metric $\gK$ induced by $g$ and a parameter function $f(\ta)$,
%as in Theorem \ref{ad-gen}.
its Ricci form and scalar curvature are given, respectively, by
\[
\rho_{\scriptscriptstyle{K}}=
-\fr 12dJd\log (-ff'r\ii/\ell),
\]
and
\be\lb{scal}
s_{\scriptscriptstyle{K}}=\fr 12((\log(f'f/\ell))'af)'/(ff')
-\ast\left[\fr 12\om\we dJd\log(-r\ii)\right],
\end{equation}
where $\om$, $\ast$ are the K\"ahler form and Hodge star operator of $\gK$, respectively,
and $a$ is the coefficient of $\kk$ in the expression for $Jd\ta$ as a
linear combination of $\kk$ and $d\ta$.
\end{prop}
Note that we could write the second term in \Ref{scal} much more explicitly
by replacing $\om$ via \Ref{decom}, but we have chosen not to, in order to keep
the formula less cluttered. Also note that the assumption on $\kk^\flat$ is
fulfilled in the case it is a connection $1$-form.
\begin{proof}
Given our assumptions on $g\big|_\HH$, the symplectic form $\om$ takes the form
\[
\om=f'd\ta\we \kk^\flat+fd\kk^\flat=f'd\ta\we \kk^\flat-f\ii r\,dx\we dy,
\]
where we dropped pull-backs by $\pi$ from the notation.
%with $\xx,\yy$ the usual orthonormal frame of $g|_\HH$.
Thus the volume form is
\be\lb{vol}
\mathrm{Vol}=\om^2/2=-ff'r\ii \, d\ta\we \kk^\flat\we dx\we dy,
\end{equation}
and $\mu=-ff'r\ii$.

% we first assume that
%$\xx^\flat$, $\yy^\flat$
%are given locally as exact $1$-forms, denoted by abuse of notation $d\xx$, $d\yy$,
%such that $z=\xx+i\yy$ forms a complex coordinate in some chart.
Next, our assumptions, taken together with Proposition \ref{holo}, mean that
$\Xi:=\kk-i\tT$ is a holomorphic vector field and $\psi:=\kk^\flat+i\tT^\flat$
a holomorphic $1$-form which evaluates to a nonzero constant on it. Employing
$\psi$ together with $dz$ one computes the coordinate volume form coefficient to find,
as $\tT^\flat=\ell d\ta$ is exact because $\ell$ is a function of $\ta$, that up to a
multiplicative constant, one can take $\nu=\ell$ (note that the assumption on
$\kk^\flat$ is used in this step).

Substituting in \Ref{ric} gives the formula for the Ricci form.

To compute the scalar curvature, we first note that $Jd\tau=J\tT^\flat/\ell$ is a
linear combination of $\kk^\flat$ and $d\ta$, as it vanishes on $\HH$.
Substituting $\kk$ and $\tT$ in $Jd\ta=a\kk^\flat+bd\ta$ and solving the
linear system for $a$, $b$, one easily sees, as the metric $g$ has constant values
when evaluated on  pairs from $\{\kk,\tT\}$, that $b$ is constant, while $a$,
which also depends on $\ell$, is a function of $\ta$. Thus
\[
\begin{aligned}
\rho_{\scriptscriptstyle{K}}&=
-\fr 12dJd\log (-ff'r\ii/\ell)=-\fr 12d[(\log(ff'/\ell))'Jd\tau] - \fr 12dJd\log (-r\ii)\\
&=-\fr 12d[(\log(ff'/\ell))'(a\kk^\flat+bd\ta)] - \fr 12dJd\log (-r\ii)\\
&=-\fr 12d[(\log(ff'/\ell))'a\kk^\flat] - \fr 12dJd\log (-r\ii)\\
&=-\fr 12((\log(ff'/\ell))'a)'d\ta\we\kk^\flat - \fr 12(\log(ff'/\ell))'ad\kk^\flat - \fr 12dJd\log (-r\ii)
\end{aligned}
\]
Computing now $\om\we\rho_{\scriptscriptstyle{K}}$, the contribution from the
first two terms in the last line above combines
to be
\[
\fr 12((\log(f'f/\ell))'af)'r\ii d\ta\we \kk^\flat\we dx\we dy
\]
Since the scalar curvature is $s_{\scriptscriptstyle{K}}=\ast(\om\we\rho_{\scriptscriptstyle{K}})$, and $\ast\mathrm{Vol}=1$, the result follows.
\end{proof}

A special case of this formula for the scalar curvature occurs when $\ell=1$ and $f(\tau)=\tau_c$,
as then ($a$ is constant and hence) the first term in \Ref{scal} vanishes, giving
\[
s_{\scriptscriptstyle{K}}=-\ast\big[\fr 12\om\we dJd\log(-r\ii)\big].
\]
%In the (unrealistic) case where $\HH$ is integrable one can take $r=1$ and then
%the scalar curvature vanishes whenever $\ii$ is constant.

\vspace{.1in}
\begin{prop}
Let $(M,g)$ be admissible with admissible complex structure $J=\J$.
%inducing a K\"ahler metric $\gK$ as in Theorem \ref{ad-gen}.
Suppose all assumptions of Proposition \ref{rc-scal1} hold,  except that $\kk$ is not
geodesic of constant  length, but rather Killing.
Assume $g(\kk,\tT)=0$, $q:=g(\tT,\tT)\ne 0$ and $p:=g(\kk,\kk)$, $\ii$
are, in addition to $\ell$, functions of $\ta$.
Then for any K\"ahler metric $\gK$ induced by $g$ and a parameter function $f(\ta)$,
its Ricci form and scalar curvature are given, respectively, by
\[
\rho_{\scriptscriptstyle{K}}=
-\fr 12dJd\log (-f\,(f'+f\, p'/p)r\ii p/\ell),
\]
and
\be\lb{scal1}
\begin{aligned}
s_{\scriptscriptstyle{K}}&=\fr {(fa P')'+2fa P'p'/p} {f\,(f'+f\, p'/p)}
-\fr 12\ast[\om\we dJd\log r],\\[2pt]
%&\ \mathrm{where}\ a=-q/(p\ell),\quad P=-\log (-f\,(f'+f\, p'/p)p\,\ii/\ell)/2,
\end{aligned}
\end{equation}
where $a=-q/(p\ell)$, $P=-\log (-f\,(f'+f\, p'/p)p\,\ii/\ell)/2$
and $\ast$ is the Hodge star operator of $\gK$.
\end{prop}
We omit the proof, which is somewhat more complex than the previous one but follows
the same outline.

\section{Admissible Lorentzian metrics inducing SKR metrics}\lb{skr}

In this section we begin the study of examples.
We first recall the description of a special type of K\"ahler metric, called SKR, also known
as a metric admitting a special K\"ahler-Ricci potential. These include many
conformally-Einstein K\"ahler metrics. We then produce from such a metric, via an explicit ansatz,
Lorentzian admissible metrics admitting a Killing field, which in some cases is also geodesic of constant length. These Lorentzian metrics in turn induce the initial SKR metric.

\vspace{.1in}
A Killing potential $\ta$ on a K\"ahler manifold $(M,J,\gK)$ is, by definition, a smooth function $\ta$
such that $J\nK\ta$ is a Killing vector field, where $\nK$ denotes the Levi-Civita connection
(or the gradient) with respect to the K\"ahler metric.
We set
\[
v:=\nK\ta,\quad u:=Jv,\quad \VV:=\mathrm{span}(v,u),\quad \HH:=\VV^\perp.
\]
This potential $\ta$ is called a special K\"ahler-Ricci potential, and $\gK$ an SKR metric,
if $\ta$ is nonconstant, and at each regular point of $\ta$, the nonzero tangent
vectors in $\HH$ are eigenvectors of both the Ricci endomorphism and the Hessian of $\ta$.
%We refer to pairs (g, τ ) for which these conditions are satisfied as an SKR structure.
We will often denote such metrics by $\gS$. As mentioned above, they include many K\"ahler conformally
Einstein metrics in dimension four  (and all of those, in higher dimensions).

\vspace{.1in}
Theorem 18.1 in \cite{dm1} gives the local classification of SKR metrics. It states
that for any SKR metric, every regular point of $\ta$ has a neighborhood $U$
which is the domain of a biholomorphic isometry $\Psi$ to an open set in a holomorphic line bundle over a K\"ahler manifold $(N, h)$ with K\"ahler form $\om^h$, equipped with the following metric, still denoted $\gS$. There are in fact two metric forms, but we only give one, which we will call
the irreducible form, as it describes metrics which are not K\"ahler local products. It is
given as follows.
\be\lb{skr-form}
\text{$\gS$ is\qquad $\fr 1Qd\ta^2+\fr Q{a^2}{\hat{u}}^2$ on $\VV$, \quad  $2|\tc|\pi^*h$ on $\HH$.}
\end{equation}
where $\tc:=\ta-c$ is as in Section \ref{sec:prop} with $c$ a constant, $\ta$ is the push-forward
of the Killing potential under the above biholomorphism, $a\ne 0$ is a constant,
$Q$ is a function of $\ta$ which equals $\gS(v,v)=\gS(u,u)$, $\pi$ is the projection map from the
line bundle to $N$, $\hat{u}$ is the one-form having value $a$ on
$u$ and zero on $v$ and on lifts of vector fields on $N$, and $\VV$, $\HH$ are also obtained via
pushing-forward the same-named distributions via the biholomorphism, with $\VV$ being also the vertical distribution of the line bundle. In addition, $\HH$ is the horizontal distribution for
a Chern connection on the line bundle.

\vspace{.1in}
If $M$ is compact and not biholomorphic to $\mathbb{CP}^m$, it follows from \cite{dm2}
that a biholomorphic isometry $\Psi$ as above exists, with domain $M$, mapping onto a
$\mathbb{CP}^1$-bundle over a K\"ahler manifold equipped with a canonical model metric.
Furthermore, $\Psi$ maps the non-critical set of $\ta$ onto the total space of a
line bundle minus its zero section. Finally, in the irreducible case the model metric still has
the form \Ref{skr-form} on this subset of the total space.

\vspace{.1in}
We record some known relations for SKR metrics, all immediate from or appearing in \cite{dm1},
some of which will be employed below. In these, $w$, $w'$ denote horizontal lifts of vector fields
on the base manifold $N$.
\be\lb{SKR-rel}
\begin{aligned}
&i)\, \gS(u,v)=0,\qquad ii)\, \text{$Q>0$ if $v\ne 0$ or $u\ne 0$,}\qquad iii)\, [u,v]=0,\\
&iv)\, [v,w]=0,\qquad v)\, [u,w]=0,\qquad vi)\, [w,w']^\VV=-2\pi^*\om^h(w,w')u.\\
\end{aligned}
\end{equation}

Given an SKR metric in dimension four of the form \Ref{skr-form}, our
main objective is to obtain a procedure for finding an admissible Lorentzian
metric $g$, %with a Killing vector field $\kk$,
defined on an appropriate subset $U$, with an induced K\"ahler metric satisfying
$\gK=\gS$ on $U$.
%We do so under some conditions.

\vskip 6pt
The construction is as follows. Set $\kk:=u$, $\tT:=-v$.
Fix two of the three values of the metric on $\kk$, $\tT$
as follows: $g(\kk,\tT):=0$, while $g(\tT, \tT):=q$ is defined to be
an arbitrarily chosen negative constant.
%Set $\ell:=-q/Q$,
%which is well defined away from the critical set of $\ta$.
Choose $p$ to be a function of a variable $\ta$ which is positive
on $\{\ta>c\}$.
%and satisfy $fp=\tc$.
Set $g(\kk,\kk):=p$, with $p$
now abusively denoting $p\circ\ta$. Define $g\big|_\VV$ by linear
extension. Then define $g\big|_\HH:=\pi^*h$. Finally, declare $g(\VV,\HH)=0$.
%Next, Let $s$ be the scalar curvature of $\gS$,
%which is often a function of $\ta$. When this is so,
%define $\ii$ to be a solution to the second order ode
%\be\lb{ode}
%s(\tc\,(1+\tc\, p'/p))=(a\tc P')'+2a\tc P'p'/p
%\end{equation}
%with $a$, $P$ defined as in \Ref{scal} and independent variable denoted
%by abuse of notation $\ta$. If $\ii>0$ on the interval $\{\ta>c\}$ lying in the range
%of the function $\ta$, define there
%$g\big|_\HH=2\pi^*h/\ii$, with $\ii$ now considered a
%function on our manifold via composition with $\ta$.

In the following theorem we will take the domain of the biholomorphic
isometry $\Psi$ to be the entire manifold under discussion, with the
range contained in the total space of the line bundle. Our theorem is
then stated as follows.
\begin{thm}\lb{SKRthm}
%Suppose the scalar curvature $s$ of $\gS$ is a function
%of $\ta$ such that there exists a solution $\ii$ of
%the ode \Ref{ode} which is positive whenever $\ta>c$.
Let $\gS$ be an irreducible SKR metric on a complex manifold $(M,J)$
of real dimension four with special K\"ahler-Ricci potential $\ta$,
such that the biholomorphic isometry $\Psi$ has domain $M$.
Then there exists a Lorentzian metric on
$U:=\{d\ta\ne 0\}\cap\{\ta>c\}$,
which is isometric, via $\Psi$, to an admissible metric $g$ among those in the
ansatz just described, whose distinguished vector fields are $\kk$ and $\tT$. %Furthermore, $\tT=\ell\n\ta$ for $\ell:=-q/Q$,
%while $\kk$ is Killing, and also geodesic of constant length if
%$p$ is a positive constant.

Choosing $f(\ta):=\tc/p$, the metric $g$ along with $f$ induce a
K\"ahler metric $\gK$ on the line bundle, %as in Theorem \ref{ad-gen},
whose isometric copy in $M$ (also denoted $\gK$) is defined on $U$ and satisfies
\[
\text{$\gK=\gS$ on $U$.}
\]
If $M$ is compact and not biholomorphic to $\mathbb{CP}^2$,
then (after perhaps switching the sign of $\ta$) this
Lorentzian metric is in fact defined on the set $U := \{d\ta \ne 0\}$,
which is open and dense in $M$.
\end{thm}
\begin{proof}
We identify $U$ from now on with its image in the line bundle.
It is easy to see that $g$ is Lorentzian and conditions
\Ref{nnsing} and \Ref{space} hold.

%Note first that on $U$ the vector fields $v$, $u$ have
%no zeros, hence the same holds for $\kk,\tT$, so that their
%assigned lengths via $g$ are well defined. As $u=Jv$, the fields
%$\kk$, $\tT$ are linearly independent at each point of $U$.

%\vspace{.08in}
%On $U$, or even on $\{\ta>c\}$, $p$ is positive.
%Thus  $g(\kk,\kk)>0$, while $g(\tT,\tT)=q<0$ and $g(\kk,\tT)=0$.
%Hence $g\big|_\VV$ is nondegenerate of index one.
%On the other hand $g\big|_\HH=\pi^*h$ is positive definite.
%Since $g(\HH,\VV)=0$, $g$ is a Lorentzian metric at each point of $U$.
%For this metric $\HH$ is spacelike, and $\kk$, $\tT$ are everywhere linearly
%%independent, so that
%%$G=pq\ne 0$ on $U$, so that
%conditions \Ref{nnsing} and \Ref{space} hold.

\vspace{.08in}
Conditions \Ref{Nij}i) follow immediately from \Ref{SKR-rel}iv),v),
the definition of $\kk$, $\tT$ and the fact that in the SKR setting,
horizontal lifts of base vector fields span $\HH$.

\vspace{.08in}
%As $\kk$ is $g$-Killing, it is shear-free. Hence
Condition ii) of \Ref{Nij} follows as the Koszul formula,
applied to the expressions in \Ref{eqn:shear2}, shows that
$\tT$ and $\kk$ are shear-free.
%But if
%$\xx$, $\yy$ are horizontal lifts of base vector fields forming an
%ordered orthonormal frame on $\HH$, then
%$2\sig_1^\tT=g(\n_{\!\yy}\tT, \yy)-g(\n_{\!\xx}\tT, \xx)=0$ by the
%Koszul formula, as the values of $g$ on all pairs of vector fields
%taken from $\xx$, $\yy$, $\tT$ are constant and the Lie brackets $[\xx,\tT]$,
%$[\yy,\tT]$ are zero. A similar calculation shows $\sig_2^\tT=0$, where this
%time the only nonzero terms involves the brackets $[\xx,\yy]$ and $[\yy,\xx]$,
%which cancel each other. The same argument works for $\kk$.

\vspace{.08in}
Having checked i] of Definition \ref{adms} of admissibility, we now
turn to ii]. %In fact we will show the stronger near-gradient relation
%\Ref{near-grad}i).
First, note that setting $\ell:=-q/Q$ is well defined on $U$
as $Q$ is nonzero there.
%Second, this $\ell$ is a function of $\ta$,
%so as $d_{\xx}\ta=\gS(\xx,\n^{\scriptscriptstyle{S\!K\!R}}\ta)=\gS(\xx,v)=0$
%for $\xx$ lying in $\HH$, it follows that $\ell$ has a gradient lying in $\VV$.
Now define $\tilde{\tT}=\tT/\ell$ (note that $\ell\ne 0$).
We wish to show $\tilde{\tT}=\n\ta$. The values of $g$ on the pairing of either of these
vector fields with $\xx$, $\yy$, $\kk$ is zero, since, for example
$g(\tilde{\tT},\xx)=g(\tT,\xx)/\ell=0=d_{\xx}\ta=g(\xx,\n\ta)$, and similarly
$g(\tilde{\tT},\kk)=0$, while we have seen that $0=d_\kk\ta=g(\kk,\n\ta)$. Their values
when paired with $\tT$ are compared as follows: $g(\tilde{\tT},\tT)=q/\ell$ while
$g(\n\ta,\tT)=d\ta(\tT)=\gS(\n^{\scriptscriptstyle{S\!K\!R}}\ta,\tT)=\gS(v,\tT)=\gS(v,-v)= -Q=q/\ell$. This shows $\tT=\ell\n\ta$.%, verifying \Ref{near-grad}i).

%\vspace{.08in}
%Next, $g(\kk,\tT)$ is constant by the very definition of $g$.
%The mixed condition \Ref{sm-Rm} holds because for any horizontal lift $\xx$
%of a vector field on $N$,
%$g(\n_{\!\kk}\tT+\n_{\!\tT}\kk,\xx)=g([\kk,\tT]+2\n_{\!\tT}\kk,\xx)=2g(\n_{\!\tT}\kk,\xx)
%=-[g(\tT,[\kk,\xx])+g(\kk,[\tT,\xx])-g(\xx,[\tT,\kk])]=0$ via the Koszul
%formula, where the last equality holds because the Lie brackets in all
%three terms are zero.

\vspace{.08in}
To complete the proof that $g$ is admissible we need to verify iii] of Definition~\ref{adms}.
Now $g(\kk,\tT)$ is constant by the very definition of $g$. Also,
$\n(g(\kk,\kk))\in\Gamma(\VV)$. This follows because $\n\kk$ defines a map
$\HH\to\kk^\perp$, as can be seen since for a vector field $\xx$ with values
in $\HH$, $2g(\n_{\!\xx}\kk,\kk)=d_{\xx}(g(\kk,\kk))=d_{\xx}p=p'd_{\xx}\ta=0$.

\vspace{.1in}
Next, we check that with $f$ given in the theorem, the induced $\gK$ is
a K\"ahler metric on $U$. Note first that $q/\ell=-Q<0$ on $U$, while our $f$
is well defined and positive on $U$ and satisfies $f'p+f p'=(fp)'=\tc'=1$.
Next, from \Ref{SKR-rel}iii) we see that $[\kk,\tT]=0$, so that
minus the expression \Ref{alt}, which is just the left hand side of \Ref{ineq}
as $g(\kk,\tT)=0$, evaluates to $f'G/\ell+f d_\tT(g(\kk,\kk))=(q/\ell)(f'p+f p')=-Q$,
which is negative on $U$.
Also, applying \Ref{SKR-rel}vi) and $\kk=u$, we see that for the usual horizontal
lifts $\xx$, $\yy\!\!=\!\!J\xx$ we have
%of an ordered $h$-orthonormal frame,
$\ii=g(\kk,[\xx,\yy])=-2\pi^*\om^h(\xx,\yy)p=-2\pi^*h(\yy,\yy)p=-2p<0$ on $U$. Hence
$f\ii<0$ on $U$. Thus, by Theorem~\ref{ad-gen}, $g$ induces
a metric $\gK$ on $U$ which is K\"ahler with respect to $J$, as clearly $J=\J$.

\vspace{.08in}
We now show that $\gK=\gS$ on $U$. In fact, $\gK(\tT,\tT)=\gK(\kk,\kk)=-(f'G/\ell+f d_\tT p)=Q$
as we have just seen, i.e. this value is equal to $\gS(-v,-v)=\gS(u,u)$, while $\gK(\kk,\tT)=0=\gS(u,-v)$.
Thus $\gK\big|_\VV=\gS\big|_\VV$. On the other hand,
$\gK\big|_\HH=-f\ii g\big|_\HH=-f(-2p)\pi^*h=2\tc\pi^*h=2|\tc|\pi^*h=\gS\big|_\HH$ on $U$.
As $\gK(\VV,\HH)=0=\gS(\VV,\HH)$, the first equality following from Lemma \ref{basic},
the two K\"ahler metrics indeed coincide.

\vspace{.08in}
Finally, if $M$ is compact and not biholomorphic to $\mathbb{CP}^2$, we
know that the non-critical set of $\ta$ is mapped via $\Psi$ onto the line bundle minus
its zero section, and we wish to show the former set is $U$.
It was shown in \cite{dm2} that for an SKR metric on a compact manifold, the
range of $\ta$ is a closed interval and $c$ is not in its interior.
As choosing the sign of $\ta$ determines that of $c$, one can thus
always arrange, when the manifold is compact, that $\ta>c$ everywhere. Then,
from $U$'s definition
%because the
%$\ta$-critical submanifolds are exactly the level sets of $\ta$ corresponding to the
%endpoints of the range of $\ta$ (\cite{dm2}),
it follows that $U$ is exactly
the non-critical set of $\ta$.
As $\ta$ is a Killing potential, it is
known that the latter set is open and dense in $M$ (cf. \cite{dm2}).

\vspace{.08in}
This completes the proof.
\end{proof}
Note that it is possible to show that $\kk$ is $g$-Killing and,
if $p$ is a positive constant, also $g$-geodesic, while $\tT$ is always
$g$-geodesic.
%Observe also that as all assumptions of Proposition \ref{Killing} hold,
%it provides an alternative route for showing that $\kk$ is $\gK$-Killing,
%independent from the recognition that $\gK=\gS$.

\vspace{.08in}
It is worth mentioning that one can modify the ansatz slightly so that
%if in the ansatz above we make $g(\tT,\tT):=q$
%a {\em positive} constant, while $g(\kk,\kk):=p$ a function of $\ta$ which is
%{\em negative} on $\{\ta>c\}$, the above proof goes through, and
the Lorentzian metric will admit a
%of Theorem \ref{SKRthm} goes
%through without change, except that $\iota$ will now be positive and $f$
%negative on that set, hence we will still have $f\imath<0$. Thus we
%can choose our Lorentzian metrics inducing SKR metrics to have a
{\em timelike Killing field}.

%%\vspace{.08in}
%%A final point to note is that in the case where $\gS$ is defined
%%on a compact manifold %(not biholomorphic to $\mathbb{CP}^2$),
%%the metric $g$ will be defined only on an open dense set.
%given via the ansatz on an open dense set $U$, does not extend to $M$. The
%reason is that $M\setminus U$ consists of zeros of $\tT$ and $\kk$, but $g(\tT,\tT)$
%is a negative constant on $U$, so cannot become zero smoothly on $M\setminus U$.
%This, of course, goes along with the fact that the expression \Ref{symp-expl}
%for the exact K\"ahler form cannot hold on the entire compact manifold,
%even though the K\"ahler form does extend smoothly.

%---------------------------------------------------------
\section{K\"ahler metrics induced by Lorentzian warped products}
\label{warp}

\subsection{The construction}
In this section we construct K\"ahler 4-manifolds from \emph{admissible} Lorentzian 4-manifolds, in the sense of Definition \ref{adms}, which are warped products.
%In \cite{am2} we describe
%K\"ahler-Einstein metrics that arise from this construction.
%Let $(\RR \times N,-dt^2+w^2\bar{g})$ be a Lorentzian warped product, with $(N,\bar{g})$ a Riemannian 3-manifold and $w$ a function of $t$.
%equipped with a unit length vector field $\bar{\kk}$ whose flow is geodesic and shear-free, and whose orthogonal complement $\kk^{\perp} \subset TM$ is nowhere integrable (i.e., the flow of $\bar{\kk}$ is twisting everywhere; recall Lemma \ref{lemma:Ray}).
%The corresponding null-symplectic 4-manifold, in the sense of Definition \ref{defn:nullsymp}, will be the \emph{warped product}
%$$
%(\,\RR \times_f M\,,\,\underbrace{-dt + f^2\bar{g}}_{g}\,),
%$$
%where $f$ is a smooth positive function of $t$.  The corresponding null vector field will be
%$$
%\kk:=\partial_t + \frac{\bar{\kk}}{f},
%$$
%where $\bar{\kk}$ is as mentioned above, while the requisite time function is the coordinate function $t$.  We then provide an explicit example in the case when $(M,\bar{g})$ is the round 3-sphere $(\mathbb{S}^3,\bar{g})$.
%In what follows, $\nabla$ will denote the Levi-Civita connection of the warped product, and $\overline{\nabla}$ that of $(N,\bar{g})$.

\begin{thm}
\label{thm:warped}
Let $(N,\bar{g})$ be a Riemannian 3-manifold with a unit length vector field $\bar{\kk}$, whose flow is geodesic, shear-free, and has a nowhere vanishing twist function.  Let $w(t)$ be a smooth positive function on $\RR$ satisfying $w'/w > -1$. Then
$(\RR \times N,g,\kk,\nabla t)$ is admissible
with respect to a chosen orientation, where $g$
is the Lorentzian warped product
$$
g := -dt^2 + w^2\bar{g}
$$
and $\kk:=\partial_t + \bar{\kk}/w$. The metric $g$ then
induces a K\"ahler metric on $\RR \times N$.
\end{thm}
Note that in the expression for $\kk$, the notation $\bar{\kk}$ refers to the obvious lift
of this vector field from $N$ to $\RR \times N$.
\begin{proof}
We verify the admissible properties of Definition \ref{adms}.
To begin with, our vector field $\kk$ is clearly $g$-null, and pre-geodesic:
\be
\cd{\kk}{\kk} = \frac{w'}{w}\,\kk.\label{eqn:kpreg}
\end{equation}
%\beqa
%\cd{\kk}{\kk} &=& \cd{\partial_t}{(\partial_t + \bar{\kk}/w)}+\frac{1}{w}\cd{\bar{\kk}}{(\partial_t + \bar{\kk}/w)}\nonumber\\
%&=& \underbrace{\cd{\partial_t}{\partial_t}}_{0}\ -\ \frac{w'}{w^2}\,\bar{\kk}+\frac{1}{w}\underbrace{\cd{\partial_t}{\bar{\kk}}}_{\frac{w'}{w}\bar{\kk}}+\frac{1}{w}\underbrace{\cd{\bar{\kk}}{\partial_t}}_{\frac{w'}{w}\bar{\kk}}\ -\ \frac{1}{w^3}\underbrace{\bar{\kk}(w)}_{0}\bar{\kk}+\frac{1}{w^2}\underbrace{\cd{\bar{\kk}}{\bar{\kk}}}_{ww'\partial_t}\nonumber\\
%&=& \frac{w'}{w}\left(\partial_t + \frac{\bar{\kk}}{w}\right)=\frac{w'}{w}\,\kk.\label{eqn:kpreg}
%\eeqa
(For the properties satisfied by the Levi-Civita connection of warped products, see \cite[Proposition 35, p. 206]{o1983}.)
%Here we have used the fact that $\cd{\bar{\kk}}{\bar{\kk}}$ has $\RR$-component $\text{nor}\,\cd{\bar{\kk}}{\bar{\kk}} = -\frac{g(\bar{\kk},\bar{\kk})}{w}\,\nabla w = ww'\partial_t$ and $N$-tangential component $\text{tan}\,\cd{\bar{\kk}}{\bar{\kk}} = \overline{\nabla}_{\!\bar{\kk}}{\bar{\kk}} = 0$ (see, e.g., \cite[Proposition 35, p. 206]{o1983}).
Next, set $\tT := \nabla t = -\partial_t$ and observe that
\[
g(\tT,\tT) = -1,\quad \cd{\tT}{\tT} = 0,\quad g(\kk,\tT) = \kk(t) = 1.
\]
We now establish relations between the shear, and later also twist functions of $\kk$ and $\bar{\kk}$. In what follows, the shear coefficients and twist function (up to sign) of $\bar{\kk}$ will be denoted by $\bar{\sigma}_1,\bar{\sigma}_2$ and $\bar{\ii}$, respectively, while those of $\kk$ will be denoted by $\sigma_1,\sigma_2$ and $\ii$.

Let $\{\bar{\xx},\bar{\yy}\}$ be an ordered local $\bar{g}$-orthonormal frame of $\bar{\kk}^{\perp_{\bar{g}}} \subset TN$, whose ordering is chosen so that $\bar{\ii}$ is {\em negative}.
Lifting these vector fields trivially to $\RR \times N$, the vector fields $\xx := \bar{\xx}/w$ and $\yy := \bar{\yy}/w$ form a  $g$-orthonormal frame for the (spacelike) distribution $\HH = \mathrm{span}(\kk,\tT)^{\perp_g}$. Furthermore, $\cd{\xx}{\kk}=\frac{1}{w^2}\big(w'\bar{\xx} + \overline{\nabla}_{\!\bar{\xx}}{\bar{\kk}}\big)$, $\cd{\yy}{\kk} = \frac{1}{w^2}\big(w'\bar{\yy} + \overline{\nabla}_{\!\bar{\yy}}{\bar{\kk}}\big)$,
%\beqa
%\label{eqn:shear1}
%\cd{\yy}{\kk} = \frac{1}{w^2}\big(w'\bar{\yy} + \overline{\nabla}_{\!\bar{\yy}}{\bar{\kk}}\big) %\comma \cd{\xx}{\kk}=\frac{1}{w^2}\big(w'\bar{\xx} + %\overline{\nabla}_{\!\bar{\xx}}{\bar{\kk}}\big),
%\eeqa
%\beqa
%\cd{\yy}{\kk} &=& \frac{1}{w}\cd{\bar{\yy}}{(\partial_t + \bar{\kk}/w)}\nonumber\\
%&=& \frac{1}{w}\bigg(\underbrace{\cd{\bar{\yy}}{\partial_t}}_{\frac{w'}{w}\bar{\yy}}+\frac{1}{w}\cd{\bar{\yy}}{\bar{\kk}}\bigg)\nonumber\\
%&=& \frac{w'}{w^2}\bar{\yy}+\frac{1}{w^2}\bigg(\underbrace{\text{tan}\,\cd{\bar{\yy}}{\bar{\kk}}}_{\overline{\nabla}_{\!\bar{\yy}}{\bar{\kk}}}+\underbrace{\text{nor}\,\cd{\bar{\yy}}{\bar{\kk}}}_{0}\bigg)\nonumber\\
%&=& \frac{1}{w^2}\big(w'\bar{\yy} + \overline{\nabla}_{\!\bar{\yy}}{\bar{\kk}}\big),\label{eqn:y1}
%\eeqa
%where we have used the fact that $\text{nor}\,\cd{\bar{\yy}}{\bar{\kk}} = -\frac{g(\bar{\yy},\bar{\kk})}{w}\,\nabla w = 0$.  Likewise,
%\[
%\cd{\xx}{\kk}=\frac{1}{w^2}\big(w'\bar{\xx} + \overline{\nabla}_{\!\bar{\xx}}{\bar{\kk}}\big).
%\]
where $\overline{\nabla}$ is the Levi-Civita connection of $\bar{g}$.  That $\kk$ is shear-free now follows from this and the fact that $\bar{\kk}$ is shear-free in $(N,\bar{g})$:
%\[
%2\bar{\sigma}_1=\bar{g}(\bar{\yy},\overline{\nabla}_{\!\bar{\yy}}\bar{\kk}) - \bar{g}(\bar{\xx},\overline{\nabla}_{\!\bar{\xx}}\bar{\kk})=0
%\hspace{.1in},
%\hspace{.1in}
%2\bar{\sigma}_2=\bar{g}(\bar{\xx},\overline{\nabla}_{\!\bar{\yy}}\bar{\kk}) + \bar{g}(\bar{\yy},\overline{\nabla}_{\!\bar{\xx}}\bar{\kk})=0.
%\]
\be
2\sigma_1=g(\nabla_{\!\yy}\kk,\yy) - g(\nabla_{\!\yy}\kk,\xx)=\frac{2}{w}\,\bar{\sigma}_1=0,\nonumber%\\
%2\sigma_2\!\!&=&\!\!g(\nabla_{\!\yy}\kk,\xx) + %g(\nabla_{\!\xx}\kk,\yy)=\frac{2}{w}\,\bar{\sigma}_2=0.\nonumber
\end{equation}
%Equivalently,
%$$
%\nabla^o \kk = \begin{bmatrix}%{cc}
%        - \sigma_1 & \sigma_2\\
%        \sigma_2 & \sigma_1\\
%      \end{bmatrix} = O
%$$
%(recall the frame representation of the shear operator in \eqref{shr0}).
and a similar relation for $\sigma_2$.
It is likewise verified that $\tT$ is also shear-free, since $\cd{\xx}{\tT} = -(w'/w)\xx$ and $\cd{\yy}{\tT} = -(w'/w)\yy$. Thus, the almost complex structure $J := J_{g,\kk,\tT}$
compatible with the orientation for which
$J\xx:=\yy, J\yy:=-\xx$,
satisfies the shear condition
$J\nabla^o\kk = \nabla^o\tT=0$.
%Next, note that
%$$
%g([\kk,\cdot],\kk)\big|_{\HH} = g([\kk,\cdot],\tT)\big|_{\HH} = 0
%$$
%by \eqref{double} and \eqref{near-grad} in Remark \ref{cond-i}, while $g([\tT,\cdot],\kk)\big|_{\HH} = 0$ because $[\tT,\xx] = [\tT,\yy] = 0$.

Furthermore, as $\cd{\kk}{\tT} = -\frac{w'}{w^2} \bar{\kk}$ and $\cd{\tT}{\kk} = 0$, condition \eqref{semRm} also holds.
%, hence $\VV=\mathrm{span}(\kk,\tT)$ satisfies \Ref{sem-Rm}.
Thus all conditions of subsection~\ref{geo-cond} hold, so that
\Ref{alter} and therefore \Ref{Nij}i) hold. Thus by Theorem \ref{integ}, $J$ is integrable.
Also all conditions in Definition \ref{adms} hold, so $(\RR \times N,g,\kk,\tT)$ is therefore admissible.

To complete the proof, we apply Theorem \ref{ad-gen} to show that $g$ induces a K\"ahler metric on $\RR \times N$. Recalling the definition of $\ii$ in \eqref{eqn:twist0}, observe that (up to sign), $\ii$ of $\kk$ in $(\RR \times N,g)$ is related to that of $\bar{\kk}$ in $(N,\bar{g})$ by
\begin{equation}\lb{warped-twist}
\ii = g(\nabla_{\!\yy}\kk,\xx) - g(\nabla_{\!\xx}\kk,\yy)=\frac{1}{w}\,\bar{\ii} < 0.
\end{equation}
%Furthermore, we have chosen $\bar{\xx},\bar{\yy}$ on $N$ so that the (nowhere vanishing) twist %function $\bar{\ii}$ defined with respect to it is negative; then $\ii$ in \eqref{warped-twist} %will be negative.
Now choose $f(t) = e^t$.  Since $w'/w > -1$ and $G = -\ell = -1$,
\[
f\ii = e^t \ii < 0,\qquad f'G/\ell-f(w'/w) g(\kk,\tT) = -e^t(1+ (w'/w))<0,
\]
so that $\gK = d(e^t\kk^\flat)(\cdot,J\cdot)$ will be a K\"ahler metric on all of $\RR \times N$.
%Indeed, in the frame $\{\kk,\xx,\yy,\tT\}$, $\omega$ takes the form
%\beqa
%\omega=e^t\left[\begin{array}{cccc}
%0 & 0 & 0 & -1\\
%0 & 0 & -\ii_g & 0\\
%0 & \ii_g & 0 & 0\\
%1 & 0 & 0 & 0
%\end{array}\right],\nonumber
%\eeqa
%so that $\gK$ takes the form
%\beqa
%\gK=e^t\left[\begin{array}{cccc}
%1 & 0 & 0 & 0\\
%0 & \ii_g & 0 & 0\\
%0 & 0 & \ii_g & 0\\
%0 & 0 & 0 & 1
%\end{array}\right]\cdot\nonumber
%\eeqa
%This completes the proof.
\end{proof}
In~\cite{am2} it is shown that (for other functions $f(t)$) one can produce
K\"ahler-Einstein metrics from this construction.

By Lemma \ref{lemma:Ray}, this result immediately yields:
\begin{cor}
\label{cor:Ric}
Let $(N,\bar{g})$ be a Riemannian 3-manifold and $X$ a unit length vector field whose flow is complete, geodesic, and shear-free.  If \emph{$\text{Ric}_{\bar{g}}(X,X) > 0$}, then $\RR \times N$ admits a K\"ahler metric as in Theorem \ref{thm:warped}.
\end{cor}

We now present some concrete realizations of Theorem \ref{thm:warped}.

%\section{K\"ahler metrics on product manifolds $\RR \times M$}
%\label{sec:examples}
\subsection{A K\"ahler metric via the direct product on $\RR \times \SSt$}
Let $\bar{g}$ denote the standard round metric on the 3-sphere $\SSt$.
%; i.e., the metric induced on $\SSt$ from the Euclidean metric on $\RR^4$.
On $\SSt$, let $\bar{\kk}$ denote the unit length Killing vector field tangent to the Hopf fibration.
%In coordinates $(x^1,y^1,x^2,y^2) \in \RR^4$, it is the restriction to $\SSt$ of the following vector field on $\RR^4$:
%$$
%\sum_{i = 1}^2-y^i \frac{\partial}{\partial x^i} + x^i \frac{\partial}{\partial y^i}\cdot
%$$
%On, say, the upper hemisphere $(y^2 > 0)$ of $\SSt$, $\bar{\kk}$ takes the form
%$$
%\bar{\kk}=-y^1\frac{\partial}{\partial x^1} + x^1\frac{\partial}{\partial y^1} - \sqrt{1-(x^1)^2-(y^1)^2-(x^2)^2}\,\frac{\partial}{\partial x^2}\cdot
%$$
By Lemma \ref{lemma:KVF}, the flow of $\bar{\kk}$ is geodesic and shear-free in $(\SSt,\bar{g})$, and its twist function $\ii$ satisfies
$$
\ii^2=2\text{Ric}_{\bar{g}}(\bar{\kk},\bar{\kk})=2.
$$
Applying Theorem \ref{thm:warped} with $w = 1$, $\kk = \partial_t+\bar{\kk}$, and $\tT = \nabla t$, we conclude that $(\RR \times \SSt,-dt^2 \oplus \bar{g},\kk,\tT)$ is admissible and induces a K\"ahler metric on $\RR \times \SSt$. It will turn out that this K\"ahler metric, defined with $f(t)=e^t$, is flat (see \cite{am2}).

%----------------------------------------------------------------
\subsection{A K\"ahler metric on de Sitter spacetime}
Four-dimensional \emph{de Sitter spacetime} is the warped product $(\RR \times \SSt_r,-dt^2 + w^2\bar{g})$, where $w(t) = r^2 \cosh^2 (t/r)$ and $r > 0$ is the radius of $\SSt_r$.
This metric is both globally hyperbolic and geodesically complete (see \cite[p. 183-4]{beem96}).
For $r \geq 2$, $w'/w > -1$ on the entire manifold (if $r < 2$, then on an open subset). Thus, with $\kk = \partial_t + \bar{\kk}/w$ and $\tT = \nabla t$, Theorem \ref{thm:warped} applies.

%%----------------------------------------------------------------
%\subsection{The twist of the Sasaki metric in the product case}
%
%Let $(N,g:=-dt^2\oplus h)$,
%where $h$ is a metric on a $3$-manifold $M$,
%with null geodesic vector field $\kk=\partial_t+\kr$, with
%$\kr$ denoting both a vector field on $M$ and its lift to $N$.
%Let $\xx$, $\yy$ be an orthonormal basis for $g\big|_\HH$, where
%$\HH=\mathrm{span}(\kr,\partial_t)^{\perp}$.
%Then $\xx$, $\yy$ are also $h$-orthonormal and span $\kr^\perp_h$.
%Also $[\xx,\yy]_{\scriptscriptstyle{N}}=[\xx,\yy]_{\scriptscriptstyle{M}}$ is tangent to $M$. Therefore,
%the twist of $\kr$ in $g$ satisfies
%\[
%\ii_{\kr,g}=g(\kr,[\xx,\yy]_{\scriptscriptstyle{N}})=h(\kr,[\xx,\yy]_{\scriptscriptstyle{M}})=\ii_{\kr,h}.
%\]
%Expressing $[\xx,\yy]|_M=\alpha\kr+\beta x+\gamma y$ and recalling
%that $\kr$ has unit length in $h$, we can also write
%$\ii_{\kr,h}=\al h(\kr,\kr)=\al$.

%-----------------------------------------------
\subsection{K\"ahler metrics on $\RR \times \RR^3$}
\label{sec:ppwave}
Using Theorem \ref{thm:warped} once again, a family of K\"ahler manifolds will now be constructed on $\RR^4$ out of the following distinguished class of Lorentzian 4-manifolds:

\begin{defn}
\label{defn:ppwave}
A four-dimensional \emph{standard pp-wave} is the Lorentzian manifold $(\RR^4,g)$ with coordinates $(u,v,x,y)$ and with $g$ given by
\be
\label{eqn:stdpp}
g=H(u,x,y)du \otimes du + du \otimes dv + dv \otimes du + dx \otimes dx + dy \otimes dy,
\end{equation}
for $H$ smooth. If $H$ is quadratic in $x$ and $y$, then $(\RR^4,g)$ is called a
\emph{plane wave}.
\end{defn}

Standard pp-waves originated in gravitational physics and have been intensely studied therein; see, e.g., \cite{Leistner} and \cite{Costa}, as well as \cite{Sanchez} and \cite[Chapter 13]{beem96}.
%With respect to $\{\partial_u,\partial_v,\partial_x,\partial_y\}$ the matrices of $g$ and its inverse are given by
%\beqa
%(g)_{ij}=\left[\begin{array}{cccc}
%H & 1 & 0 & 0\\
%1 & 0 & 0 & 0\\
%0 & 0 & 1 & 0\\
%0 & 0 & 0 & 1
%\end{array}\right]\hspace{.2in},\hspace{.2in}(g)^{ij}=\left[\begin{array}{cccc}
%0 & 1 & 0 & 0\\
%1 & -H & 0 & 0\\
%0 & 0 & 1 & 0\\
%0 & 0 & 0 & 1
%\end{array}\right]\cdot\nonumber
%\eeqa
They are distinguished by the fact that $\partial_v = \nabla u$ is a parallel null vector field.
%Furthermore, for any choice of the function $H(u,x,y)$, the corresponding pp-wave is scalar flat, and the only nonvanishing component of its Ricci tensor is
%$$
%\text{Ric}(\partial_u,\partial_u)=-\frac{1}{2}(H_{xx} + H_{yy}),
%$$
%where $H_{xx}$ and $H_{yy}$ denote the second partial derivatives of $H$ with respect to $x$ and $y$, respectively.
Now let $k,h\colon \RR^4\lra \RR$ be two smooth functions independent of $v$ and consider the following null vector field $\ZZ$ on $(\RR^4,g)$:
\be
\label{eqn:ell}
\ZZ:=\frac{1}{2}\left(H + k^2 + h^2\right)\!\partial_v - \partial_u + k\partial_x + h\partial_y.
\end{equation}
Observe that $\partial_v$ and $\ZZ$ are pointwise linearly independent, that $g(\partial_v,\ZZ) = -1$, and finally that the pair of vector fields
\be
\label{eqn:xy3}
\xx:=k\partial_v + \partial_x\hspace{.2in},\hspace{.2in}\yy:=h\partial_v + \partial_y
\end{equation}
are orthonormal and span the distribution $\HH = \text{span}(\partial_v, \ZZ)^{\perp}$, which is spacelike. %(in the literature, such bundles are usually referred to as \emph{screen distributions}; see, e.g., \cite{Leistner}).
Note also that the twist function of $\ZZ$ does not, in general, vanish, as $\ii^{\ZZ}$ is
\beqa
\label{eqn:elltwist0}
g(\ZZ,[\xx,\yy])=g(\ZZ,(-k_y + h_x)\partial_v)=k_y - h_x.
\eeqa
Because $\partial_v = \nabla u$, the hypersurfaces $S_u := \{u = \text{const.}\}$ are integral submanifolds of the orthogonal complement $\partial_v^{\perp} \subset T\RR^4$; because $\partial_v$ is null, it is tangent to these submanifolds.
Now fix any $u_0$ and consider the hypersurface $S_{u_0} \cong \RR^3$ with global coordinates $\{v,x,y\}$.  Set
\be
\label{eqn:kS}
\bar{\kk}:=\partial_v|_{S_{u_0}}\hspace{.2in},\hspace{.2in}
\bar{\xx}:=\xx|_{S_{u_0}}\hspace{.2in},\hspace{.2in}\bar{\yy}:=\yy|_{S_{u_0}}
\end{equation}
and define a Riemannian metric $\bar{g}$ on $S_{u_0}$, by giving its orthonormal coframe $\{\bar{\kk}^{\bar{\flat}},\bar{\xx}^{\bar{\flat}},\bar{\yy}^{\bar{\flat}}\}$, which is $\bar{g}$-dual to $\{\bar{\kk},\bar{\xx},\bar{\yy}\}$. It is given by
\vspace{.05in}
\beqa
\label{eqn:h}
\bar{g}(\bar{\kk},\cdot):=-g(\ZZ,\cdot)\big|_{S_{u_0}}\hspace{.1in},\hspace{.1in}
\bar{g}(\bar{\xx},\cdot):=g(\xx,\cdot)\big|_{S_{u_0}}
\hspace{.1in},\hspace{.1in}\bar{g}(\bar{\yy},\cdot):=g(\yy,\cdot)\big|_{S_{u_0}},
\eeqa
so that $\{\bar{\kk},\bar{\xx},\bar{\yy}\}$ is a global orthonormal frame for $(\RR^3,\bar{g})$.  (This Riemannian metric is derived from a well-known construction; see, e.g., \cite{Leistner}. Note that
it does not depend  on $H$.) We now have:
\begin{prop}
\label{eqn:KVFproof}
On $(\RR^3,\bar{g})$ with $\bar{g}$ given by \eqref{eqn:h},
the vector field $\bar{\kk}$ given in \eqref{eqn:kS} is a unit
length Killing vector field.  If $h_x-k_y$ is nowhere
vanishing and $w$ is a smooth positive function satisfying
$w'/w > -1$, then $(\RR \times \RR^3,-dt^2 + w^2\bar{g},\bar{\kk}/w+\partial_t,\nabla t)$
is admissible and  induces a K\"ahler metric on $\RR^4$.
\end{prop}

\begin{proof}
We first show that $\bar{\kk}$ is a unit length Killing vector field, via Lemma \ref{lemma:KVF}.  That $\bar{\kk}$ has unit length with respect to $\bar{g}$ is clear, since
$\bar{g}(\bar{\kk},\bar{\kk}) = -g(\ZZ,\bar{\kk}) = 1$.  Letting $\overline{\nabla}$ denote the Levi-Civita connection of $(\RR^3,\bar{g})$, it follows in particular that $\bar{g}(\overline{\nabla}_{\!\bar{\kk}}\bar{\kk},\bar{\kk}) = 0$.  Next, observe that
\[
\bar{g}(\overline{\nabla}_{\!\bar{\kk}}\bar{\kk},\bar{\xx})=-\bar{g}(\bar{\kk},[\bar{\kk},\bar{\xx}])=
g(\ZZ,[\partial_v,\xx])\big|_{S_{u_0}}=0,
\]
where we have used $[\bar{\kk},\bar{\xx}] = [\partial_v,\xx]\big|_{S_{u_0}}$ because $[\partial_v,\xx]$ is tangent to $S$, and also $[\partial_v,\xx] = 0$.  Likewise, $\bar{g}(\overline{\nabla}_{\!\bar{\kk}}\bar{\kk},\bar{\yy}) = 0$, so that $\overline{\nabla}_{\!\bar{\kk}}\bar{\kk} = 0$, hence $\bar{\kk}$ has geodesic flow in $(\RR^3,\bar{g})$.  That $\bar{\kk}$ is both divergence-free and shear-free is similarly determined.
%That the divergence of $\kk$ vanishes is similarly determined.
%Indeed,
%$$
%\bar{g}(\overline{\nabla}_{\!\bar{\xx}}\kk,\bar{\xx})=-\bar{g}(\bar{\xx},[\kk,\bar{\xx}])=-g(\xx,[\partial_v,\xx])\big|_{S_{u_0}}=0.
%$$
%Likewise with $\bar{g}(\overline{\nabla}_{\!\bar{\yy}}\kk,\bar{\yy}) = 0$, showing that the divergence of $\kk$ with respect to $\bar{g}$ vanishes:
%$$
%\text{div}_{\bar{g}}\,\kk=\bar{g}(\overline{\nabla}_{\!\kk}\kk,\kk) + \bar{g}(\overline{\nabla}_{\!\bar{\xx}}\kk,\bar{\xx}) + \bar{g}(\overline{\nabla}_{\!\bar{\yy}}\kk,\bar{\yy})=0.
%$$
%So, too, does its shear: $2\sigma_1 = \bar{g}(\overline{\nabla}_{\!\bar{\yy}}\kk,\bar{\yy}) - \bar{g}(\overline{\nabla}_{\!\bar{\xx}}\kk,\bar{\xx}) = 0$, and the Koszul formula on $\sigma_2$ yields
%\beqa
%2\sigma_2=\bar{g}(\overline{\nabla}_{\!\bar{\yy}}\kk,\bar{\xx}) + \bar{g}(\overline{\nabla}_{\!\bar{\xx}}\kk,\bar{\yy})\!\! &=&\nonumber\\
%&&\hspace{-1.25in} \frac{1}{2}\Big(\!\!-\bar{g}(\kk,[\bar{\yy},\bar{\xx}])-\bar{g}(\bar{\xx},[\kk,\bar{\yy}])+\bar{g}(\bar{\yy},[\bar{\xx},\kk])\Big)\nonumber\\
%&&\hspace{-.75in} +\ \frac{1}{2}\Big(\!\!-\bar{g}(\kk,[\bar{\xx},\bar{\yy}])-\bar{g}(\bar{\yy},[\kk,\bar{\xx}])+\bar{g}(\bar{\xx},[\bar{\yy},\kk])\Big)\nonumber\\
%&=& -\bar{g}(\bar{\yy},[\kk,\bar{\xx}]) + \bar{g}(\bar{\xx},[\bar{\yy},\kk])\nonumber\\
%&=& -g(\yy,[\partial_v,\xx])\big|_{S_{u_0}} + g(\xx,[\yy,\partial_v])\big|_{S_{u_0}}\nonumber\\
%&=& 0.\nonumber
%\eeqa
Being unit length, geodesic, divergence-free, and shear-free, it follows by Lemma \ref{lemma:KVF} that $\bar{\kk}$ is a unit length Killing vector field on the Riemannian 3-manifold $(\RR^3,\bar{g})$.  Finally, we show that the twist function of $\bar{\kk}$ is nowhere vanishing, provided that $k_y \neq h_x$ at any point. This follows from our particular choice of $\ZZ$, namely, that its twist function is nowhere vanishing:
\beqa
\label{eqn:htwist}
\bar{\ii} := \bar{g}(\bar{\kk},[\bar\xx,\bar\yy])=-g(\ZZ,[\xx,\yy])\big|_{S_{u_0}}=(h_x - k_y)\big|_{S_{u_0}}\!.
\eeqa
Applying now the contents of Theorem \ref{thm:warped}, the proof is complete.
\end{proof}

Some of the K\"ahler metrics induced by such ``truncated" pp-wave metrics as in this proposition,
as well as some of those induced directly from plane waves (see next section), are shown in
\cite{am2} to be central metrics \cite{c}. More precisely, the determinant of their Ricci endomorphism vanishes. Other such metrics are shown there to be K\"ahler-Einstein.

\subsection{Complete induced K\"ahler metrics}\lb{complt}

We show here that one can in some instances obtain complete K\"ahler
metrics via the construction of Theorem~\ref{thm:warped}.
%To this end we employ
%the following proposition, communicated by A. Derdzinski, whose proof is an application of the
%Hopf-Rinow Theorem.
%\begin{prop}\lb{compl}
%Given a complete submanifold $N$ of a Riemannian manifold $(M,h)$, closed
%as a subset of $M$, assume geodesics normal to $N$ are complete,
%while minimizing geodesics emanating from $N$ in normal
%directions cover $M$. Then $M$ is complete.
%\end{prop}
%We note that if $M$ is, for example, a fiber bundle, the above condition on
%the minimizing geodesics will hold if $N$ is a section and
%the fibres are orthogonal to $N$, complete and totally geodesic.

To proceed we choose a {\em compact} $3$-manifold $N$
with the properties in Theorem~\ref{thm:warped} and consider on
$M=N\times \mathbb{R}$ a K\"ahler metric $\gK$
induced according to that theorem but defined via an arbitrary  function $f$.
Its domain is then computed (via Theorem~\ref{ad-gen}, the sign information
for $\bar{\iota}$ and $w$ along with \Ref{warped-twist},
Remark~\ref{LIK} and \Ref{eqn:kpreg}) to be
\be\lb{domain}
f>0,\qquad (fw)'/w>0.
\end{equation}

Let $\hat{\kk}$ denote the $1$-form that is $1$ on $\kk$ and zero on the other
frame fields, and have $\hat\tT$, $\hat\xx$, $\hat\yy$, where $\xx$, $\yy$ form an oriented
orthonormal frame for $\HH$, denote analogous $1$-forms. Similarly denote by $\hat{\bar\kk}$
the pull-back from $N$ of the one-form dual in the same sense to $\bar{\kk}$.
Now $dt=\hat\kk-\hat\tT$, while $(\hat\kk+\hat\tT)(\bar{\kk}/w)=2$ so that the
K\"ahler metric takes the form
\[
\begin{aligned}
\gK&=c(\hat\kk^2+\hat\tT^2)-f\ii(\hat\xx^2+\hat\yy^2)\\
&=(c/2)((\hat\kk-\hat\tT)^2+(\hat\kk+\hat\tT)^2)-f\ii(\hat\xx^2+\hat\yy^2)\\
&=(c/2)(dt^2+(2w\hat{\bar\kk})^2)-fw^{-1}\bar{\ii}(\hat\xx^2+\hat\yy^2)
\end{aligned}
\]
for $c=\gK(\kk,\kk)=(fw)'/w$.
It follows that $\gK$ has the form $ds^2+g_s$ where $s=\int\sqrt{c/2}\,dt$
and $g_s$ form a family of metrics on $N$. As $N$ is compact, such metrics are complete
on $I\times N$ whenever $I=(\inf s,\sup s)=\mathbb{R}$ (cf. \cite[Lemma 33]{a-z}).
The latter can easily be arranged with an appropriate choice of $f$ and $w$.

%---------------------------------------------------------
\section{A K\"ahler metric induced by a gravitational plane wave}
\label{sec:planewave}

In Proposition \ref{eqn:KVFproof} we used three-dimensional submanifolds of pp-waves
(Definition~\ref{defn:ppwave}) to construct K\"ahler metrics on admissible four-dimensional Lorentzian warped products.  Key to this construction was the null vector field \eqref{eqn:ell} of a pp-wave, which was used to define the Riemannian metric \eqref{eqn:h}.  In this example we will construct a K\"ahler metric which will be induced directly from an admissible $4$-dimensional pp-wave (in fact, a plane wave),
%recall Definition \ref{defn:ppwave}),
without first constructing a Riemannian metric on a hypersurface.

\begin{prop}
\label{prop:ppKahler}
Let $(\RR^4,g)$ be a pp-wave with $H(u,x,y) = -x^2 - y^2$ and let $\ZZ$ be the null vector field \eqref{eqn:ell} with $k(u,x,y) = -y$ and $h(u,x,y) = x$.  Then $(\RR^4,g,\ZZ,\nabla u)$ is admissible and induces a K\"ahler metric on $\RR^4$.
\end{prop}
\begin{proof}
$\ZZ=-\partial_u-y\partial_x+x\partial_y$ and $\n u=\partial_v$ are everywhere linearly independent and the orthonormality of $\xx$, $\yy$ of \Ref{eqn:xy3} shows that
$\HH:=\mathrm{span}(\n u, \ZZ)^\perp$ is spacelike, as required in \Ref{space}.
For $\kk_+:=\ZZ$ and the parallel field $\kk_-:=\n u$, \eqref{double}a) holds, as well
as \eqref{near-grad} and condition \eqref{semRm}, the latter since $\cd{\partial_v}{\ZZ} = \cd{\ZZ}{\partial_v} = 0$. Thus by Remark~\ref{cond-i}, \Ref{Nij}i) holds, as can also be
seen directly from the bracket relations of these vector fields with $\xx$, $\yy$, which also
imply using \Ref{eqn:shear2} that $\ZZ$, $\n u$ are shear-free. Thus the associated almost complex structure is integrable. Next, $g(\ZZ,\ZZ)=0$ and $g(\n u, \ZZ)=-1$, so
$(\RR^4,g,\ZZ,\nabla u)$ is admissible.
%\beqa
%g(\partial_v,[\ell,\xx]) \!&=&\! g(\partial_v,\cd{\ell}{\xx} - \cd{\xx}{\ell})=0,\nonumber\\
%g(\partial_v,[\ell,\yy]) \!&=&\! g(\partial_v,\cd{\ell}{\yy} - \cd{\yy}{\ell})=0,\nonumber\\
%g(\ell,[\ell,\xx]) \!&=&\! g(\ell,\cd{\ell}{\xx} - \cd{\xx}{\ell})=g(\cd{\ell}{\ell},\xx)=0,\nonumber\\
%g(\ell,[\ell,\yy]) \!&=&\! g(\ell,\cd{\ell}{\xx} - \cd{\xx}{\ell})=g(\cd{\ell}{\ell},\yy)=0.\nonumber
%\eeqa
%(The other inner products vanish because $[\partial_v,\xx] = [\partial_v,\yy] = 0$, where we recall that these brackets vanish for any $f,h$, since the latter are not functions of $v$.)  Thus with integrability.
%That
%$$
%\omega=d(e^uZ^\flat)
%$$
%is a symplectic form follows from Proposition \ref{sympl}: with $f = e^u$ and $\ii = f_y - h_x = -2$,
%$$
%ff'\ii = -2e^{2u} \neq 0,
%$$
%as required.
%Indeed, with respect to the frame $\{\partial_v,\xx,\yy,Z\}$,
%\beqa
%\omega=e^u\left[\begin{array}{cccc}
%0 & 0 & 0 & -1\\
%0 & 0 & 2 & 0\\
%0 & -2 & 0 & 0\\
%1 & 0 & 0 & 0
%\end{array}\right]\cdot\nonumber
%\eeqa
Now set $f = e^{u}$ and note that $\ii^{\ZZ} = k_y - h_x = -2$.  That $\gK=d(e^u\ZZ^\flat)(\cdot, J\cdot)$ is K\"ahler on $\RR^4$ now follows from Theorem~\ref{ad-gen} and Remark~\ref{LIK}, as
%, with $G = -1$ and $\ell = 1$:
$f\ii = -2e^{u} < 0$ and $f'G/\ell = -e^u < 0$.
\end{proof}
It can be shown that $\ZZ$ is $\gK$-conformal (in fact $\gK$-homothetic) while $\n u$ is $\gK$-Killing. It is shown in \cite{am2} that $\gK$ has vanishing Ricci determinant.
\section{K\"ahler metrics on Petrov type~$D$ spacetimes}
\label{Kerr}
In this section we present two examples of Lorentzian 4-manifolds that fall outside the admissible category, and a third admissible one.
These examples belong to the class of Lorentzian 4-manifolds of \emph{Petrov Type~$D$,} and include the \emph{Kerr spacetime.}
%\vspace{.08in}
The relationship between such spacetimes and K\"ahler geometry
was considered by Flaherty \cite{flaherty}, who was lead to construct
mostly modified-K\"ahler metrics living on the complexified tangent bundle,
rather than genuine K\"ahler metrics.

\vspace{.08in}
%Another difference with \cite{flaherty}, is that our construction of an associated
%K\"ahler metric requires one of the vector fields to have nowhere vanishing twist with
%respect to the Lorentzian metric.

\subsection{K\"ahler metrics for Lorentzian metrics of Petrov type~$D$}
\label{Dtype}

By the Goldberg-Sachs Theorem \cite{GS}
(see also \cite[Chapter 5]{o1995}), a Lorentzian metric of
Petrov type~$D$ admits two null geodesic vector fields which
are both shear-free. Even when these satisfy the conditions of
Theorem \ref{integ}, neither one may be gradient, or even
near-gradient in the sense of Definition~\ref{adms}(ii).
Thus they do not give rise to an admissible manifold as in Definition \ref{adms}.
%and one cannot form with them K\"ahler metrics with symplectic form \Ref{symp-expl}.
However, if one takes $\kk_+$ to be one of these geodesic fields,
and $\kk_-$ a pre-geodesic field which is a function multiple of the second,
it is still possible to form very similar K\"ahler
metrics, as the following proposition shows:
\begin{prop}\lb{for-Kerr}
Let $g$ be a Petrov type~$D$ metric on an oriented $4$-manifold,
with null shear-free vector fields $\kk_\pm$, with $\kk_+$
geodesic, $\kk_-$ pre-geodesic and $g(\kk_+,\kk_-)<0$.
Assume also that $J:=\JJ$ is an admissible complex structure.
%the conditions $[\kk_\pm,\Gamma(\HH)]\subset\Gamma(\HH)$
% the Lie bracket conditions \Ref{Nij}i)
%of Theorem \ref{integ}
%hold with $\HH=\mathrm{span}(\kk_+,\kk_-)^\perp$.
Set $p:=1/\sqrt{-g(\kk_+,\kk_-)}$.
Suppose $u$ is a smooth function on $M$ and $f$
a smooth positive function defined on
the range of $u$, for which $\n(f(u)/p)\in\Gamma(\VV)$ for
$\VV=\mathrm{span}(\kk_+,\kk_-)$.
%If $J:=\JJ$ is the admissible complex structure
%determined by $\kk_\pm$,
Then
\[
\gK=d(f(u)p\kk_+^\flat)(\cdot,J\cdot)
\]
is K\"ahler in any region where $\,\ii^{\kk_+}<0$ and $d_{\kk_+}(f(u)p)<0$.
%, with $\ii^{\kk_+}$ computed as in Theorem \ref{ad-gen}.
\end{prop}
%Note that if $p=1$ above, this becomes essentially Proposition \ref{prop:DgK}.
\begin{proof}
Setting $\kt_\pm:=p\kk_\pm$, note that $g(\kt_+,\kt_-)=-1$ and $\n_{\!\kt_+}\!\kt_+=(d_{\kk_+}p)\kt_+$, so that $\kt_+$ is null pre-geodesic.
Observe also that $J_{g,\kt_\pm} = \JJ$. Now consider $\om:=d(f(u)\tilde{\kk}_+^\flat)$.  For the usual orthonormal frame $\xx,\yy$ on $\HH$, it follows from a calculation as in Theorem~\ref{ad-gen} that
$\om(\kt_+,\xx) = \om(\kt_+,\yy) = 0$, and similarly $\om(\xx,\yy)=-fp\ii^{\kk_+}$, and
$$
\om(\kt_+,\kt_-)=-\big(f'du(\kt_+)+fd_{\kk_+}p\big).
$$
%Similarly, we also see that $\om(\kt_+,\kt_-)=-(f'du(\kt_+)+fd_{\kk_+}p)$.
%, as can be verified
%by breaking $\om$ into two terms as usual and using $g(\kt_+,\kt_-)=-1$.
%Next, for the usual orthonormal frame $\xx$, $\yy$  on
%$\HH$ we have $\om(\xx,\yy)=fd\kt_+^\flat(\xx,\yy)=-fp\ii^{\kk_+}$.
Thus $\om$ will be symplectic if $f(f'du(\kt_+)+fd_{\kk_+}p)\ii^{\kk_+}\ne 0$.

For $J$-invariance of $\om$, again the crucial test is when one vector field lies in $\VV$
and the other in $\HH$. Since we know $\om$ vanishes on the null pre-geodesic vector field $\kt_+$ and
a vector field in $\HH$, we need to show the same on $\kt_-$ and such a vector field.
As in Theorem \ref{ad-gen}, we calculate, noting that $g([\kk_-,\xx],\kk_+)=0$ as usual,
and relying on the constancy of $g(\kt_+,\kt_-)$:
\[
\begin{aligned}
\om(\kt_-,\xx)&=-f'du(\xx)g(\kt_+,\kt_-)+f(-g([\kt_-,\xx],\kt_+))\\
&=f'du(\xx)+fp(d_\xx p)g(\kk_-,\kk_+)=f'd_\xx u-f(d_\xx p)/p.
\end{aligned}
\]
Vanishing of this follows if $d_\xx (\log f(u))=d_\xx(\log p)$.
Vanishing for any $\xx$ with values in $\HH$ will thus hold if
$\log(f(u)/p)$, or equivalently $f(u)/p$, has a vertical gradient.

Tameness of $J$ in the region specified by the inequalities in the
theorem follows from the above calculations of $\om(\xx,\yy)$ (note
that $f$, $p$ are positive) and
$\om(\kt_+,\kt_-)$, where for the latter we note that
$-(f'du(\kt_+)+fd_{\kk_+}p)=-(f'p\,du(\kk_+)+fd_{\kk_+}p)=-d_{\kk_+}(f(u)p)$
must be positive.
%if $d_{\kk_+}(\log(f(u)p))<0$, because $f$ is positive.
\end{proof}
Note that $\kt_\pm$ satisfy the conditions of Theorem~\ref{thm:KerrNUT},
rather than Theorem~\ref{integ}.
\begin{remark}\lb{doubKJ}
Remark~\ref{doubleJ} applies here,
while Remark~\ref{doubleKJ} follows analogously, so
that $\gK$ of Proposition~\ref{for-Kerr} is ambihermitian.
\end{remark}

\subsection{A K\"ahler metric on Kerr spacetime}
\label{sec:Kerr}
The {\em Kerr spacetime} can be partitioned in the form $M=M'\cup \Sigma\cup M''$,
where $M'$, $M''$ are open submanifolds, and $\Sigma$ is a totally geodesic hypersurface
called the equatorial plane, which contains a cylinder called the ring singularity,
where the Kerr metric $g$ is singular. We construct a K\"ahler metric $\gK$ on a subset
of $M'$, in the case where the Kerr spacetime is rapidly rotating, a status determined by an
inequality between the two constant parameters in $g$. A similar K\"ahler metric can be
constructed on a subset of $M''$ if the complex structure we choose is changed by a minus sign
on a distribution denoted $\HH$ and described below.

\vspace{.1in}
The Kerr metric $g$ is initially defined in an open subset of $M:=\RR^2 \times \mathbb{S}^2$.
In coordinates $\{t,r,\vartheta,\varphi\}$, with $0 < \vartheta < \pi$ and
$0 \leq \varphi \leq 2\pi$, $g$ has components
\beqa
\label{eqn:Kerr}
g_{tt}=-1 + \frac{2mr}{\rho^2}\hspace{.2in},\hspace{.2in}g_{rr} &=&
\frac{\rho^2}{\Delta}\hspace{.2in},\hspace{.2in}g_{\vartheta\vartheta}=\rho^2,\\
g_{\varphi\varphi}=\left[r^2 +a^2
+\frac{2m r a^2 \sin^2\vartheta}{\rho^2}\right]\sin^2\vartheta&,&g_{\varphi t}
=g_{t \varphi}=-\frac{2mra\sin^2\vartheta}{\rho^2},\nonumber
\eeqa
all other components being zero, with $a$, $m$ positive parameters and
\beqa
\label{eqn:Kerr2}
\rho^2:=r^2 + a^2\cos^2\vartheta\hspace{.2in},\hspace{.2in}\Delta:=r^2 -2mr + a^2.
\eeqa
%We observe in passing that while the limiting case $a=0$, in which one recovers
%the \emph{Schwarzschild spacetime}, is a warped product with fiber a $2$-sphere, the Kerr metric
%itself is not warped in this manner.
Kerr spacetime is designated as rapidly rotating
if $a > m$,
%where the parameter $m$ corresponds to the mass, and $a$ to the
%angular momentum per unit mass of the spherical body or black hole being modeled
%by the Kerr metric.
and in that case $\Delta$ has no real roots.
Our choice of the rapidly rotating version is made
for convenience only, to simplify the singular region of the metric.

%On the other hand, the set of points where $\rho = 0$, namely $\{(t,0,\pi/2,\varphi)\}$,
%constitutes a genuine curvature singularity, called the \emph{ring singularity}, which
%is topologically $\Sigma := \RR \times \mathbb{S}^1$.  After a simple analytic extension
%to include $\{\theta=0\}$, $\{\theta=\pi\}$, Kerr spacetime will be defined on
%$\RR^2 \times \mathbb{S}^2 - \Sigma$, endowed with an extension of the metric above.
%We refer the reader to \cite{o1995} for more on the geometry of the Kerr metric.
%\vskip 12pt
As a (Ricci-flat) Petrov Type~$D$ metric,
%(as defined in \cite{petrov00}).  As mentioned earlier, by the Goldberg-Sachs Theorem \cite{GS} %(see also \cite[Chapter 5]{o1995}), this implies that it has %precisely
$g$'s two geodesic and shear-free null vector fields $\kk_{\pm}$ are everywhere linearly
independent and given by
%In the standard
%coordinate frame $\{\partial_t,\partial_r,\partial_\vartheta,\partial_\varphi\}$, they are given %by
\beqa
\label{eqn:pnull}
\kk_\pm:=\pm\partial_r + \frac{r^2 +a^2}{\Delta}\,\partial_t + \frac{a}{\Delta}\,\partial_{\varphi}
\eeqa
(see \cite[p. 79ff.]{o1995}).
%, which are everywhere linearly independent.

We consider as usual the almost complex structure $J:=\JJ$. The orientation is fixed by choosing
%$$
%J\kk_+:=\kk_-\hspace{.2in},\hspace{.2in}J\kk_- :=-\kk_+.
%$$
on the spacelike distribution $\HH=\mathrm{span}(\kk_+,\kk_-)^{\perp}$, the orthonormal pair
\beqa
\label{eqn:sframe}
E_2:=\frac{1}{\rho} \partial_\vartheta\hspace{.2in},\hspace{.2in}E_3
:=\frac{1}{\rho\sin\vartheta}\,\partial_\varphi + \frac{a\sin\vartheta}{\rho}\,\partial_t
\eeqa
and taking
$JE_2 := E_3$, $JE_3 := -E_2$.
(The notations $E_2$, $E_3$ conform with \cite{o1995}).)
%, which we reference below.)

%As mentioned in subsection \ref{Dtype}, our setup is not admissible in the sense of Definition %\ref{adms}: neither $\kk_{\pm}$ can be expressed as proportional
%to a gradient, nor is $g(\kk_+,\kk_-)$ constant.
To construct K\"ahler metrics, we will apply Proposition \ref{for-Kerr}.
First, we verify integrability of $J$, by checking the conditions of Theorem~\ref{integ}.
Now $\kk_\pm$ are  shear-free, and conditions (i) of Theorem \ref{integ} hold because
\beqa
\label{eqn:Kbrack}
[\kk_\pm,E_2]=\mp\frac{r}{\rho^2}E_2\hspace{.2in},\hspace{.2in}[\kk_\pm,E_3]=\mp\frac{r}{\rho^2}E_3,
\eeqa
as can be easily verified using \cite[p. 95-6]{o1995}.\footnote{Note that on p. 95 of \cite{o1995}, there are two occurrences of ``$r/\sqrt{\vep\Delta}(E_0 \pm E_1)$" which should in fact be ``$\rho/\sqrt{\vep\Delta}(E_0 \pm E_1)$," with $\rho$ given by \eqref{eqn:Kerr2}.  Also, on p. 96 of \cite{o1995} the Lie bracket $[E_2,E_3]$ should equal ``$-r\sqrt{\vep\Delta}\,\rho^{-3}E_3$," not ``$-r\sqrt{\vep\Delta}\,\rho E_3$."}

Next, one easily calculates that $g(\kk_+,\kk_-)=-2\rho^2/\Delta<0$, so that
\[
p:=\sqrt{\frac{\Delta}2}\frac 1{\rho}\cdot
\]
Note that $p$ is a function of only $r$ and $\vartheta$. Let $u=e^{h(r)}p$, with
$h$ to be determined later, and $f(u)=u$, which is positive on the (positive) range of $u$.
The conditions of Proposition \ref{for-Kerr}
we want satisfied can be written
in the form
\be\lb{conds}
\begin{aligned}
&[\n\log(f(u))]^\HH=[\n\log p]^\HH,\\
%&g([\n\log(f(u))]^\VV+[\n\log p]^\VV,\kk_+)<0
&d_{\kk_+}[\log(f(u)p)]<0.
\end{aligned}
\end{equation}
Now the inverse matrix to that of $g$ has the nonzero components at the same
entries as those of $g$, and $g^{rr}=1/\rho^2>0$. Thus $\n h(r)=g^{rr}h_{r}\partial_r$,
which lies in $V$.
%(the comma denotes regular or partial differentiation).
Thus the first condition in \Ref{conds} is automatically satisfied,
since $\log f(u)=\log p +h(r)$. The second condition is just
\[
\begin{aligned}
d_{\kk_+}(h)=h_r&<-2d_{\kk_+}\log p=-(\log (p^2))_r\\
&=(\log\rho^2)_r-(\log\Delta)_r
%g^{rr}h_{r}g_{rr}&=g(\n h(r),\kk_+)<g(-2[\n\log p]^\VV,\kk_+)\\
%&=g(2[\n\log \rho]^\VV,\kk_+)-g([\n\log (\Delta/2)]^\VV,\kk_+)\\
%&=(2g^{rr}(\log \rho)_{r}-g^{rr}(\log\Delta)_{r})g_{rr}
=(2r/\rho^2-2(r-m)/\Delta).
\end{aligned}
\]
As $1/\rho^2>1/(r^2+a^2+1)$, this will be satisfied at any point of $M - \Sigma$
for which $r>0$ if
$h(r)=h_{a,m}(r)$ is given by
\[
%\begin{aligned}
h(r)=%\int^r(2x/(x^2+a^2+1)-2(x-m)/(x^2-2mx+a^2))\,dx\\
\log\left(\fr{r^2+a^2+1}{\Delta}\right).
%\end{aligned}
\]

The function $\ii$, which is, up to sign, the twist function of $\kk_+$,
is given (cf. \cite{o1995}) in the ordered basis $E_2$, $E_3$, by
\beqa
\label{eqn:pcv2}
\ii=\frac{2a\cos\vartheta}{\rho^2}\cdot
\eeqa
Thus $\ii$ vanishes only on the totally geodesic hypersurface $\Sigma$ given by
$\vartheta = \pi/2$, known as the ``equatorial plane" (a plane is gotten by also fixing a value of $t$).
Hence $\ii$ will be negative ``below" it, so that
\be\lb{Mpr}
%\begin{aligned}
%\text{the }&\text{K\"ahler metric
\gK=d(e^{h(r)}p^2\kk_+^\flat))(\cdot,\JJ\cdot)=
d\left(\fr{r^2+a^2+1}{2\rho^2}\kk_+^\flat\right)(\cdot,\JJ\cdot)
\end{equation}
is a K\"ahler metric on the set
\[
M'\cap \{r>0\},\ \text{where }
M':=\{(t,r,\vartheta,\varphi) \in \RR^2 \times \mathbb{S}^2\,|\, \pi/2 < \vartheta < \pi\}.
\]
%\end{aligned}
%\end{equation}
Note that Dixon describes in \cite{dix} a K\"ahler metric which is obtained by a type of Wick rotation of the Kerr metric, also not defined on the whole spacetime, and whose K\"ahler form was given earlier in \cite{al-sc}. He also showed this metric is in fact ambitoric, and studied
its domain and asymptotic behaviour near the singular sets, including the non-rapidly rotating
case.
\subsection{A K\"ahler metric on NUT spacetime}
\label{sec:NPKerr}

One of 3 NUT spacetimes (Newman-Unti-Tamburino) is
%As noted above, Kerr spacetime is a Ricci-flat Lorentzian 4-manifold of Petrov
%type~$D$.  All such 4-manifolds have been classified in \cite{Kinne, kinn}; there are fourteen such %metrics, each containing between one and four independent parameters.  Among these are the three NUT %spacetimes (Newman-Unti-Tamburino), which were invented as generalizations of the Schwarzschild %spacetime, and are also limiting cases of the Kerr-NUT spacetimes. The NUT spacetime we consider is
given in local coordinates $\{u,r,x,y\}$ by
\beqa
\label{eqn:KerrNUT}
g_{uu}=-|\rho|^2(r^2-2mr-l^2)\hspace{.2in},\hspace{.2in}g_{ur} &=& -1\hspace{.2in},\hspace{.2in}g_{ry}=2l\cos x,\nonumber\\
g_{uy}=2|\rho|^2l\cos x (r^2-2mr-l^2) &,&g_{xx}=r^2 + l^2,\\
\hspace{.05in}g_{yy}=-|\rho|^2 (r^2-2mr -l^2)(4l^2\cos^2\!x) \!\!\!&+&\!\!\! (r^2+l^2)\sin^2\!x,\nonumber
\eeqa
all other components being zero, with $\rho := -1/(r+il)$
%$$
%\rho:=-\frac{1}{r+il}
%$$
and $l$ a positive constant (see \cite{kinn, Kinne}).

%(\eqref{eqn:KerrNUT} is equation (3.47) in \cite{kinn} with $a = 0$; note that \cite{kinn, Kinne} work %in metric index $3$, so \eqref{eqn:KerrNUT} is minus that which appears in \cite{kinn}).
NUT spacetime is Ricci-flat and topologically an open subset of $\RR^2 \times \mathbb{S}^2$.
%, with $x,y$ playing the roles that $\vartheta,\varphi$ played in Kerr spacetime, respectively.
Its two shear-free null vector fields are
\[
\kk_+:=\partial_r \comma \kk_-:=\partial_u - \frac{1}{2}|\rho|^2(r^2 - 2mr - l^2)\,\partial_r.
\]
We report without proof that the conditions of Proposition \ref{for-Kerr} hold (with $p=1$)
%can now be verified, the details of which we forego here.  Hence
so that $\gK = d(e^{-r}\kk_+^\flat)(\cdot,J\cdot)$ will be a K\"ahler metric on the entire domain of NUT spacetime lying in $S^2\times\RR^2$ (in contrast to the Kerr case).
\subsection{A conformally Kerr metric and its induced K\"ahler metric}
\label{sec:Kerr2}
Our final Petrov type~$D$ example is admissible. It is conformal to Kerr spacetime,
namely
\beqa
\label{eqn:Kerrconf}
\tilde{g}:=\frac{\Delta}{\rho^2}\,g,
\eeqa
where $g$ denotes the Kerr metric \eqref{eqn:Kerr} and the notation is as in \eqref{eqn:Kerr2}.
%, and
The associated almost complex structure $\tilde{J}:=J_{\tilde{g},\kk,\widetilde{\nabla}r}$
will be defined with respect to  $\{\kk_+,\widetilde{\n} r\}$ instead of $\{\kk_+,\kk_-\}$.
% Let $g$ denote the Kerr metric \eqref{eqn:Kerr} above, and recalling \eqref{eqn:Kerr2}, define
%\beqa
%\label{eqn:Kerrconf}
%\tilde{g}:=\frac{\Delta}{\rho^2}\,g.
%\eeqa
With this conformal factor, the gradient $\widetilde{\nabla} r = \partial_r$.
Definition \ref{adms}'s conditions may be verified to yield admissibility.
An assoicated metric of the form $\gK = d(f(r)\kk_+^\flat)(\cdot,\tilde{J}\cdot)$  is K\"ahler
on the region $M'\setminus \{r=0\}$, where $M'$ is given in \Ref{Mpr}, if
%To do so, first note that the twist function of $\kk$ is conformally invariant. Thus the corresponding
%function $\tilde{\ii}$, when computed in the ordered basis $\widetilde{E}_2, %\widetilde{J}\widetilde{E}_2$, is given by
%\eqref{eqn:pcv2}, and in particular is negative on $M'$.  Finally, setting
$f(r) = e^{-h(r)}$, for $h(r):=\int_{r_0}^r q(x)\,dx$ with $r_0$ a constant and
$q(r)= 2(r-m)/(r^2-2mr+a^2)-2/r,$ as can be verified via
Theorem \ref{ad-gen}.
% hold (note that $q(r)$ is only defined for $r\ne 0$).

%In the notation of Corollary~\ref{LIK}, we have $G = -1$, $\ell = 1$,
%$\alpha = \frac{2(r-m)}{\Delta} - \frac{2r}{\rho^2} > q(r)$ for $r\ne 0$,
%and $\tilde{\ii} < 0$ on $M'$.
%Thus
%\[
%\begin{aligned}
%&\text{$f\tilde{\ii}<0$ on $M'$}\\
%&\text{$f'G/\ell - f\alpha g(\kk,\tT) = - f'(r)-f(r)\alpha <-f'(r)-f(r) q(r)= 0$ if $r\ne 0$.}
%\end{aligned}
%\]
%It follows that $\gK$ is K\"ahler on $M'\setminus \{r=0\}$.

%--------------------------------------------------
\section{A Lie group example with vector fields which are not shear-free}\lb{sec:shrfl}

Consider the $4$-dimensional solvable real Lie algebra $\frak{g}$ with
ordered basis $\kk$, $\tT$, $\xx$, $\yy$ defined by the
following Lie bracket relations, where we list only
the non-zero ones (up to permuting entries):
\be\lb{lial}
[\kk,\xx]=\yy,\qquad [\tT,\yy]=\yy,\qquad [\tT,\kk]=\kk,\qquad [\xx,\yy]=\yy+r\kk,
\end{equation}
with $r$ a nonzero real constant. Four dimensional solvable Lie algebras have been classified
using symbolic software, cf. \cite{dgr} (note that we are choosing a slightly different
form for the bracket relations, isomorphic to one of the canonical normal forms given there).

By Lie's third fundamental theorem, there exists a Lie group $\wht{G}$
whose Lie algebra is $\frak{g}$, and one can take it to be simply connected, as we do.
The left invariant vector fields of $\wht{G}$ give a realization of this Lie algebra,
and will be denoted
% the left invariant vector fields corresponding to the above generators
by the same letters.

Define a Lorentzian metric $g$ on $\wht{G}$ by choosing an inner product on $\frak{g}$ making
the above four vectors orthonormal, with $g(\kk,\kk)=-g(\tT,\tT)=1$,
and then extending it to the tangent bundle of $\wht{G}$ as a left-invariant metric.
Then one easily checks that
\[
g([\kk,\HH],\kk)= g([\tT,\HH],\tT)= g([\kk,\HH],\tT)= g([\tT,\HH],\kk)=0,
\]
where $\HH=\mathrm{span}(\xx,\yy)$. Similarly,
%one also has
%\vspace{.1in}
%\be\lb{sher}
%\begin{aligned}
%&g([\kk,\xx],\yy)+g([\kk,\yy],\xx)=1,\qquad &&g([\kk,\yy],\yy)-g([\kk,\xx],\xx)=0,\\
%&g([\tT,\xx],\yy)+g([\tT,\yy],\xx)=0,\qquad &&g([\tT,\yy],\yy)-g([\tT,\xx],\xx)=1.
%\end{aligned}
%\end{equation}
using \Ref{eqn:shear2}, one sees that
% relations \Ref{sher} give
\[
\sig_1^\kk=-\sig_2^\tT=0,\qquad \sig_2^\kk=\sig_1^\tT=-1\ne 0,
\]
so that the shear operators of $\kk$ and $\tT$ are nonzero and $J\n^o\kk=\n^o\tT$,
where $J=\J$ is the corresponding admissible almost complex structure (with orientation compatible with the choice $\yy:=J\xx$). As the conditions of Theorem \ref{integ} hold, $J$ is integrable.
Next, the twist function of $\kk$ is
$
|\ii|=|g(\kk,[\xx,\yy])|=|r|\ne 0,
$
so that the sign of $\ii$ is fixed by the choice of $r$.
%Since $g(\kk,\kk)$ and $g(\kk,\tT)$ are constant, Theorem \ref{ad-gen}
%will show that we can define a K\"ahler metric on a region of $\wht{G}$
%to be determined, once we show  $\tT$ is a $g$-gradient.
%, in the sense of \Ref{near-grad}i).
%By \Ref{lial}, the orthogonal complement of $\tT$ is integrable.
%If the constant length vector field $\tT$ is also geodesic, then one can easily see
%that $\tT^\flat$ is closed, hence $\tT$ is locally gradient, and by simply connectedness
%of $\wht{G}$, in fact globally a gradient. Thus, it remains to show that $\tT$ is geodesic.
As $\tT$ is a left invariant vector field, $\n_\tT\tT=-\mathrm{ad}_\tT^*(\tT)$,
(see \cite[Proposition 3.18]{cheb}), where $\mathrm{ad}_\tT^*$ denotes the metric adjoint of
the differential at $\tT$ of the adjoint representation. As this differential is given by
the Lie bracket with $\tT$, it follows from \Ref{lial} that $g(\tT,[\tT,\cdot])$ vanishes
%identically
on the left invariant orthonormal frame $\kk$, $\xx$, $\yy$, $\tT$
of $\wht{G}$, so that $\n_\tT\tT=0$.
%, in the sense of \Ref{near-grad}i).
By \Ref{lial}, the orthogonal complement of $\tT$ is integrable.
As $\tT$  has constant length and is geodesic, one can easily see
that $\tT^\flat$ is closed, hence $\tT$ is locally gradient, and as
$\wht{G}$ is simply connected, in fact globally a gradient, i.e. $\tT=\n\ta$ for a function
$\ta$ on $\wht{G}$.
Given that, and the fact that $g(\kk,\kk)$ and $g(\kk,\tT)$ are constant, Theorem \ref{ad-gen}
yields induced K\"ahler metrics of the form \Ref{kler}, with K\"ahler form \Ref{symp-expl}, at
least one of which can be defined on all of $\wht{G}$. In fact,
% Thus $\tT=\n\ta$ for a
%function $\ta$ on $\wht{G}$.
%By Theorem \ref{ad-gen}, $\wht{G}$ admits K\"ahler metrics of the form \Ref{kler},
%with K\"ahler form \Ref{symp-expl}, on the region of $\wht{G}$ in which
the domain is given by $f\ii < 0$ and
$f' G/\ell - f(g(\n_\kk\kk,\tT) - d_{\tT} (g(\kk,\kk))/2) =-f'-f< 0$,
%For example, by adjusting $\ta$
%by an additive constant we can choose it to have range intersecting $(0,1)$,
so choosing $f(\ta)=e^\ta$ and
%$f(\ta)=\ta$ and
$r<0$ suffices.
%, the induced K\"ahler metric $\gK$ will be defined on the entire Lie group $\wht{G}$.

We note that while $g$ is admissible, $\kk$ is not a geometrically distinguished field in this case.

%One can check easily that $\kk$ is not geodesic, strictly pre-geodesic or Killing, so that $g$, %while admissible, does not have these geometric features, which appear regularly in many of our
%other examples.

%%%%%%%%%%%%%%
%\section*{Acknowledgements}
\bibliographystyle{siam}
\bibliography{AM-Kahler2-math-z}
\end{document}